\documentclass[a4paper,12pt,reqno]{amsart}
\usepackage{}
\usepackage{amssymb, amsmath, amsfonts, amscd}
\usepackage{mathrsfs, latexsym, array, longtable}
\usepackage[all,ps,cmtip]{xy}
\usepackage{xcolor, bm, enumitem}
\usepackage{hyperref}

\usepackage[noadjust]{cite}
\usepackage{cases}
\usepackage{enumitem}
\usepackage{indentfirst}
\usepackage{comment}
\usepackage{geometry}
\usepackage[dvipsnames]{xcolor}

\hypersetup{
	colorlinks=true,
	citecolor=orange,
	linkcolor=blue,
	urlcolor=teal,
}

\title{On Effective Iitaka Fibration Indices for Stable Minimal Models with Large Iitaka Volumes}
\author{Hexu Liu}
\date{\today}

\address{}
\email{}

\address{\rm Hexu Liu, Yau Mathematics Science Center, Tsinghua University, Haidian District, Beijing, 100084, China}
\email{liuhexu@mail.tsinghua.edu.cn}

\thanks{}

\newcommand{\roundup}[1]{\ulcorner{#1}\urcorner}
\newcommand{\rounddown}[1]{\llcorner{#1}\lrcorner}

\newcommand\vol{\text{\rm vol}}
\newcommand{\ivol}{\text{\rm Ivol}}
\newcommand{\bN}{\mathbb{N}}
\newcommand{\bZ}{\mathbb{Z}}
\newcommand{\bQ}{\mathbb{Q}}
\newcommand{\bR}{\mathbb{R}}
\newcommand{\bC}{\mathbb{C}}
\newcommand{\bP}{\mathbb{P}}

\newcommand{\lct}{\text{\rm lct}}

\newcommand{\fraction}[1]{\langle{#1}\rangle}
\newcommand{\OO}{\mathcal{O}}
\newcommand{\mult}{\text{\rm mult}}
\newtheorem{thm}{Theorem}[section]
\newtheorem{prop}[thm]{Proposition}
\newtheorem{lemma}[thm]{Lemma}

\newtheorem{defi}[thm]{Definition}
\newtheorem{eg}[thm]{Example}
\newtheorem{cor}[thm]{Corollary}
\newtheorem{cons}[thm]{Construction}

\begin{document}

\begin{abstract} 
We generalize \cite[Theorem 1.1]{CL24} to stable minimal models. Given a pair $(V,C)$ that admits a stable minimal model with fixed dimension, fixed coefficient set, and bounded relative volume, we study when the linear system $|m(K_V+C)|$ induces an Iitaka fibration, assuming the Iitaka volume of $K_V+C$ is sufficiently large.
\end{abstract}

\maketitle

\pagestyle{myheadings}
\markboth{\hfill Hexu Liu\hfill}{\hfill On Effective Iitaka Fibration Indices for Stable Minimal Models with Large Iitaka Volumes\hfill}
\numberwithin{equation}{section}

\tableofcontents
\clearpage

\section{Introduction}

Throughout this paper, we work over the complex number field $\bC$. \par

Studying pluri-canonical linear systems is an important mission in birational geometry. By \cite{HM06,Tak06,Tsu06}, for every positive integer $n$ there exists an optimal constant $r_n$ depending only on $n$ such that, for any smooth projective variety $V$ of general type and any integer $m\geq r_n$, the $m$-canonical map $\phi_{|mK_V|}$ is birational. A number of works, for example \cite{Bom73,CC15,C21}, have computed the explicit value of $r_n$ when $n$ is small. A further challenge is to establish effective birationality for varieties whose volumes are sufficiently large. By \cite[Theorem 1.1]{CL24}, we know that for any integer $n > 1$, there is a number $K>0$ such that for any smooth projective variety $V$ with $\vol(V) > K$, $m$-canonical map is birational for any $m \geq r_{n-1}$. A crucial ingredient of the proof is \cite[Theorem 1.5]{CL24}, an extension theorem for fibrations.\par 

In this paper, we extend the results of \cite{CL24} by establishing an effective Iitaka fibration result for pseudo-effective pairs that admit a stable minimal model and have sufficiently large Iitaka volumes. As a key step, we also generalize the extension theorem \cite[Theorem 1.5]{CL24}.

Here we recall the definitions of stable minimal models that will be used throughout the article; we refer to \cite{bir23b, Zhu25}. Our definition is slightly more general.

\begin{defi}\label{main-defi-for-smm}\em 
    Let $d > 0$ be an integer, $\Phi \subseteq[0,1]$ be a set, $u,v>0$ be two real numbers. A {\em klt $(d,\Phi, {\leq}u)$-stable minimal model} $(X,B),A$ consists of a klt pair $(X,B)$ and a divisor $A$ on $X$ such that
    \begin{itemize}
        \item[(1)] $X$ is of dimension $d$, and the coefficients of $B$ belong to $\Phi$; 
        \item[(2)] $K_X + B$ is semi-ample and defines a contraction $h: X \to Z$;
        \item[(3)] $A$ is an integral divisor that is big and semi-ample over $Z$ with $0 \leq \vol(A|_F) \leq u$ for a general fiber $F$ of $h:X \to Z$.
    \end{itemize}
    We say that $((X,B),A)$ is a {\em klt $(d,\Phi,{\leq}u,{\leq}v({\rm resp.} {=}v, {\geq}v))$-stable minimal model} if further
    \begin{itemize}
        \item[(4)] $\ivol(K_X + B) \leq v$ (resp. ${=}v, {\geq}v$) holds.
    \end{itemize}
    We say that $(X,B)$ {\em admits a klt $(d,\Phi, {\leq}u, {\leq}v({\rm resp. } {=}v, {\geq v}))$-stable minimal model} if there exists a klt $(d,\Phi, {\leq}u, {\leq}v({\rm resp. } {=}v, {\geq v}))$-stable minimal model $((X',B'),A')$ such that $(X', B')$ is a weak lc model of $(X,B)$ (for the definition of weak lc model, see \cite[Definition 2.1]{Bir12}).
\end{defi}

    Using the methods of \cite{BZ16,Bir21b,Zhu25}, one obtains the following result on effective Iitaka fibration indices for stable minimal models.
    
    \begin{thm}\label{main-Effective-Iitaka}
        Let $n>0$ be an integer, $\Phi \subseteq [0,1]$ be a DCC set, $u>0$ be a real number. Then there exists an integer $m_0 > 0$ such that the following holds. Assume that $(V,C)$ admits a klt $(n,\Phi,{\leq}u)$-stable minimal model {\rm (see Section \ref{stable-minimal-models} for definitions)}. Then, for any positive integer $m$ that is divisible by $m_0$, the linear system $|\rounddown{m (K_V + C)}|$ defines an Iitaka fibration. 
    \end{thm}

    For any integer $n>0$, DCC set $\Phi \subseteq[0,1]$ and real number $u >0$, we denote by $r(n,\Phi,u)$ the minimal integer $m_0>0$ such that above theorem holds. \par 
    Our main result of this paper is to prove the following theorem on stable minimal models with large Iitaka volumes. 

    \begin{thm}\label{main-for-Iitaka-large-volume}
        Let $n> 1$ be an integer and $\Phi \subseteq [0,1]$ be a DCC set, $u,\epsilon>0$ be real numbers. Then, there exists a real number $v>0$ that satisfies the following. For any $(V,C)$ that admits a klt $(n,\Phi,{\leq}u,{\geq}v)$-stable minimal model {\rm (see Section \ref{stable-minimal-models})} and $(V,C)$ is $\epsilon$-lc, for any $m>0$ such that $mC$ is integral and $m$ is divisible by $r(k,\Phi,u)$ for any $1 \leq k\leq n-1$, the linear system $|m(K_V + C)|$ induces an Iitaka fibration. 
    \end{thm}

A key step toward the above main theorem is the following extension theorem for stable-minimal-model-type and $\mathcal{M}_2(t)$-type fibrations. The relevant definitions  are given in Section \ref{Special-fibrations}.

\begin{thm}\label{main-Extention-for-2-fibers-DCC}
    Let $d>0$ be an integer, $\Phi\subseteq[0,1]$ be a DCC set, $u,v,\epsilon>0$ be real numbers. Then, there exists a real number $t>0$ that satisfies the following. Assume that
        \begin{itemize}
            \item[(1)] $f: (V,C) \to T$ admits a klt $(d,\Phi,{\leq}u, {\leq}v)$-SMM-type fibration such that $(X,C|_X)$ is $\epsilon$-lc for a general fiber $X$; {\rm (see Definition \ref{admits-klt-stable-minimal-model})}
            \item[(2)] $f: (V,C) \to T$ satisfies contidion $\mathcal{M}_2(t)$; {\rm (see Definition \ref{defi-for-C(t)})}
            \item[(3)] $(V,C)$ has a good minimal model.
        \end{itemize}
    Then, for any integer $m \geq 2$ with $mC$ being integral, there is a surjective map 
    $$H^0(V, m(K_V +C)) \to H^0(X_1, m(K_{X_1} + C|_{X_1})) \oplus H^0(X_2, m(K_{X_2} + C|_{X_2})),$$
    where $X_1,X_2$ are two different general fibers.
\end{thm}

Compared with the proof of \cite[Theorem 3.1]{CL24}, we refine and simplify our proof of Theorem \ref{main-Extention-for-2-fibers-DCC} using the following result on the restriction of the negative part of the Zariski decomposition; see Section \ref{NZ-decomp} for the definition.

\begin{thm}\label{main-Res-for-Negative-parts}
    Let $d > 0$ be an integer, $\Phi \subseteq[0,1]$ be a DCC set, $ \epsilon, u,v > 0$ be real numbers. Then, there exists a real number $t>0$ that satisfies the following. Assume that 
    \begin{itemize}
        \item[(1)] $f:(V,C) \to T$ admits a klt $(d,\Phi,{\leq}u, {\leq}v)$-SMM-type fibration such that $(X,C|_X)$ is $\epsilon$-lc for a general fiber $X$; 
        \item[(2)] $f:(V,C) \to T$ satisfies the condition $\mathcal{M}_2(t)$.{\rm (see Definition \ref{defi-for-C(t)})}
    \end{itemize}
    Then, for a general fiber $X$, we have 
    $$N_\sigma(K_V + C)|_X = N_
    \sigma(K_X + C|_X).$$
\end{thm}

The other key step is to construct such a special fibration. This idea was first proposed by McKernan in \cite{M02} and later refined by Lacini in \cite{Lac23}. Using methods similar to those in \cite{Lac23,Wang25}, we can construct a special type of fibration for pairs that admit a good minimal model and have sufficiently large Iitaka volumes.

    \begin{thm}{\label{main-construct-fibration}}
        Let $n \geq d > 1$ be integers. Let $\lambda: \bZ_{>0}  \times \bR_{>0} \to \bR_{>0}$ be a fixed function. Then, there exist $v_{d-1} \geq \cdots \geq v_1 > 0$ and $v > 0$ such that the following holds. Assume that
        \begin{itemize}
            \item[(1)] $(V,C)$ is a projective klt $\bQ$-pair that has a good minimal model;
            \item[(2)] $K_V + C$ has Iitaka dimension $d$ and $\ivol(K_V + C) > v$.
        \end{itemize}
        Then, either
        \begin{itemize}
            \item the linear system $|m'( K_V + C)|$ induces an Iitaka fibration, where $m':=\max\{m_0,2\}$, and $m_0>0$ is the minimal integer such that $m_0C$ is integral; or 
            \item after replacing $(V,C)$ with a crepant pair, there is an integer $k \in (0,d)$ and a fibration $f:(V,C) \to T$ such that for a general fiber $X$, $\dim X = n - d + k$, $f: (V,C) \to T$ satisfies condition $\mathcal{M}_2(\lambda(k,v_k))$ and $\ivol(K_X + C|_X) \leq v_k$.
        \end{itemize}
    \end{thm}

    When we consider smooth projective varieties without boundaries, the situation becomes simpler. Parallel to Theorems \ref{main-Effective-Iitaka} and \ref{main-for-Iitaka-large-volume}, we have the following results. We begin with an effective Iitaka fibration result.

    \begin{thm}\label{main-Effective-Iitaka-no-boundary}
        Let $n>0$ be an integer and $u>0$ be a real number. Then, there is an integer $s_0$ with the following property. For any smooth projective variety $V$ such that $(V,0)$ admits a klt $(n,\{0\}, {\leq}u)$-stable minimal model, the linear system $|mK_V|$ induces an Iitaka fibration for any $m \geq s_0$.
    \end{thm}

    For any integer $n>0$ and real number $u > 0$, let $s(n,u)$ denote the minimal integer $s_0>0$ for which the above theorem holds. Note that $s(n+1,u) \geq s(n,u)$.\par
    For varieties without boundaries and with large Iitaka volumes, we have the following result, which generalizes \cite[Theorem 1.1]{CL24}.
    
    \begin{thm}\label{main-no-boundary}
        Let $n>0$ be an integer and $u > 0$ be a real number. Then, there exists a real number $v > 0$ that satisfies the following. Assume that 
        \begin{itemize}
            \item $V$ is a smooth projective variety; 
            \item $(V,0)$ admits a klt $(n,\{0\},{\leq}u, {\geq}v)$-stable minimal model.
        \end{itemize}
        Then, for any integer $m>0$ such that either $m$ is divisible by $r(k,\{0\},u)$ for any $1 \leq k \leq n-1$, or $m \geq s(n-1,u)$, the linear system $|mK_V|$ induces an Iitaka fibration. 
    \end{thm}

{\bf Remarks.} Theorem~\ref{main-for-Iitaka-large-volume} requires $m$ to be divisible by $r(k,\Phi,u)$ for every $1\le k\le n-1$ for the following reason. Although examples show $r(k+1,\Phi,u)\ge r(k,\Phi,u)$ for all $k\ge 1$, we do not know whether $r(k+1,\Phi,u)$ is a multiple of $r(k,\Phi,u)$. Consequently, our statement differs slightly from \cite[Theorem~1.1]{CL24}. \par

The proof of Theorem~\ref{main-Extention-for-2-fibers-DCC} also refines the argument of \cite[Theorem 1.5]{CL24}.
Instead of using the Hacon--McKernan extension theorem \cite[Corollary 3.17]{HM06} to compare the fixed parts of the linear system before and after restriction, we control the Negative part of Zariski decomposition directly via Theorem~\ref{main-Res-for-Negative-parts}. Moreover, the pair $(V,C)$ is no longer of log general type, so Kawamata--Viehweg vanishing is unavailable; we replace it by Koll\'ar's injectivity theorem, which in turn forces us to assume the existence of good minimal models.

Theorem~\ref{main-Res-for-Negative-parts} analyses the restriction of the Zariski negative part.
The main tools are minimal-model-program techniques, boundedness of the base of a stable minimal model, and a uniform lower bound on log-canonical thresholds; the approach is inspired by \cite{BH24,Bir19,Bir21,Bir21b,Zhu25}.
\bigskip

{\bf Organization of the paper.} Section \ref{Preliminaries} gives some preliminary results. Section \ref{Stable-mm-in-famlies} gives definitions of some special fibrations and explores their properties. We also prove Theorem \ref{main-construct-fibration}. In Section \ref{Res-for-N-parts}, we prove Theorem \ref{main-Res-for-Negative-parts}. In Section \ref{Ext-thm-for-special-fibrations}, we prove Theorem \ref{main-Extention-for-2-fibers-DCC}. In Section \ref{Appications}, we prove Theorem \ref{main-Effective-Iitaka}, Theorem \ref{main-for-Iitaka-large-volume}, Theorem \ref{main-Effective-Iitaka-no-boundary} and Theorem \ref{main-no-boundary}. In Section \ref{examples}, we give some examples. 

\bigskip

{\bf Acknowledgments.}
The author expresses gratitude to Caucher Birkar and Meng Chen for great encourgement. He is grateful to Minzhe Zhu for reading a draft version of this paper and providing many valuable suggestions. He also thanks Chen Jiang, Ziqi Liu and Pengjin Wang for useful comments. 

\section{Preliminaries}\label{Preliminaries}

\subsection{DCC sets.}

\begin{defi}\em
    Let $\Phi \subseteq \bR$ be a set. We say that $\Phi$ satisfies {\em Descending chain condition} or $\Phi$ is a {\em DCC} set if $\Phi$ does not contain a strictly decreasing sequence.
\end{defi}

\subsection{Divisors and contractions.}
\begin{defi}\em

    Let $D$ be an $\bR$-divisor on a variety $X$ and $P$ be a prime divisor. Denote by $\mu_P{D}$ the coefficient of $P$ in $D$. Write  $D = \sum d_i D_i$ uniquely, where $D_i$ are different prime divisors and $d_i \in \bR$. Define the {\em roundup}, {\em rounddown} and {\em fractional part} of $D$ to be $$\roundup{D}:=\sum{\roundup{d_i}D_i}, \quad  \rounddown{D}:= \sum{\rounddown{d_i}D_i}, \quad \fraction{D} := \sum{\fraction{d_i}D_i},$$ respectively, where $\fraction{d_i} = d_i - \rounddown{d_i}$. We say that $D$ is {\em effective} and write $D \geq 0$ if $d_i \geq 0$ for all $i$. 

    Let $f: X \to Y$ be a morphism and $D$ be an $\bR$-Cartier divisor on $Y$. Sometimes we may write $D|_X$ instead of $f^*D$ for the pull-back of $D$ to $X$ if there is no confusion. For two divisors $D_1$ and $D_2$ and an integer $q>0$, we write $D_1 \sim_{q}D_2$ for $qD_1 \sim qD_2$. \par
\end{defi}

    \begin{defi}\em
        We say that a projective morphism $f:X \to Z$ between normal varieties is a {\em contraction} if $f_*\mathcal{O}_X = \mathcal{O_Z}$ holds. In particular, a contraction is always surjective with connected fibers. 
    \end{defi}

    \begin{defi}\em 
        Let $f:X \to Y$ be a contraction between normal projective varieties and $Z$ be a subvariety of $X$. We say that $Z$ is {\em horizontal} over $Y$ or {\em $f$-horizontal} if $Z$ dominates $Y$, otherwise we say $Z$ is {\em vertical} over $Y$ or {\em $f$-vertical}. 
    \end{defi}

\subsection{Linear systems.}
    \begin{defi}\em
        Let $D$ be an effective integral divisor on a normal projective variety $X$. The {\em linear system} $|D|$, the {\em $\bQ$-linear system} $|D|_{\bQ}$, the {\em $\bR$-linear system} $|D|_{\bR}$ and the {\em numerical linear system $|D|_{\rm num}$} of $D$ are respectively defined as 
        \begin{align*}
            |D|  &{:=} \{D' \geq 0 \mid D' \sim D\}, \\
            |D|_{\bQ} &{:=} \{D' \geq 0 \mid  D' \sim_{\bQ} D\}, \\
            |D|_{\bR} &{:=} \{D' \geq 0 \mid D' \sim_{\bR} D\}, \\
            |D|_{\rm num}  &{:=} \{D' \geq 0 \mid D' \equiv D\}.
        \end{align*}
    \end{defi}

\subsection{Pairs and Singularities.}
    \begin{defi}\em
        A {\em sub-pair} $(X,\Delta)$ consists of a normal quasi-projective variety $X$ and a $\bR$-divisor $\Delta$, such that $K_X + \Delta$ is $\bR$-Cartier. If additionally $\Delta \geq 0$, we say that $(X,\Delta)$ is a {\em pair}. If additionally $\Delta$ is a $\bQ$-divisor, we say that $(X,\Delta)$ is a $\bQ$-pair.\par
        
        If a pair $(X,\Delta)$ satisfies that $X$ is a smooth variety and that components of $\Delta$ has simple normal crossing support, then we say that $(X,\Delta)$ is {\em log smooth}.\par
        
    \end{defi}

    \begin{defi}\label{Pair-and-sing}\em
        Let $(X,\Delta)$ be a sub-pair. Let $D$ be a prime divisor over $X$ (i.e. $D$ is a prime divisor on a birational model of $X$). Let $f : Y \to X$ be a log resolution of $(X,\Delta)$ such that $D$ is a divisor on $Y$ and write uniquely
        $$K_Y + \Gamma = f^*(K_X + \Delta).$$
        We define the {\em log discrepancy} of the prime divisor $D$ with respect to the sub-pair $(X,\Delta)$ to be 
        $$a(D,X,\Delta) := 1 - \mu_D(\Gamma).$$
        Fix a number $\epsilon > 0$. We say $(X,\Delta)$ is {\em sub-lc} (resp. {\em sub-klt}, {\em sub-$\epsilon$-lc}) if $a(D,X,\Delta) \geq 0$ (resp. $>0$, $\geq \epsilon$) for any prime divisor $D$ over $X$. If, in addition, $(X,\Delta)$ is a pair, we say $(X,\Delta)$ is {\em lc}  (resp. {\em klt}, {\em $\epsilon$-lc}). We say $(X, \Delta)$ is {\em terminal} if $a(D,X,\Delta) > 1$ for any prime divisor $D$ that is exceptional over $X$. \par

        Let $(X,\Delta)$ be a sub-klt sub-pair, $D$ be an effective $\bR$-Cartier $\bR$-divisor on $X$ and $x \in X$ be a point (not necessarily closed). We define
        $$\lct_{x}(X,\Delta;D):= \sup \{a\in\bR \mid (X,\Delta + aD) \text{ is lc near $x$}\},$$
        and define
        $$\lct(X,\Delta;D):= \sup \{a\in\bR \mid (X,\Delta + aD) \text{ is lc everywhere}\}.$$
        Let $(X,\Delta)$ be a sub-pair and take a log resolution $f:Y \to X$. A {\em non-klt place} is a prime divisor $D$ over $X$ with $a(D,X,\Delta) \leq 0$ and a {\em non-klt center} is the image of a non-klt place on $X$. The {\em non-klt locus} of $(X,\Delta)$ is the union of all non-klt centers of $(X,\Delta)$ which is denoted as ${\rm Nklt}(X,\Delta)$.  An irreducible component of ${\rm Nklt}(X,\Delta)$
        is called a {\em maximal non-klt center}.  An {\em lc place} is a prime divisor $D$ over $X$ with $a(D,X,\Delta) = 0$ and an {\em lc center} is the image of an lc place on $X$. We say that an lc center $G$ of $(X, \Delta)$ is a {\em pure lc center} if $(X,\Delta)$ is lc at the generic point of $G$. We say that an lc center $G$ is an {\em exceptional lc center} if there is a unique lc place over $X$ whose image is $G$, and this unique lc place is called {\em the exceptional lc place}. \par
        Assume that $(X,\Delta)$ and $(X',\Delta')$ are two sub-pairs and there is a birational map $X \dashrightarrow X'$. We say that the two sub-pairs  are {\em crepant} if the pull-backs of $K_X+\Delta$ and $K_{X'}+\Delta'$ to a common birational model coincide. \par
        Assume that $(X,\Delta)$  is a pair and $\nu: X' \to X$ is a log resolution of $(X,\Delta)$. Let $\Delta' \geq 0$ be the minimal divisor such that 
        $$K_{X'} + \Delta' \geq \nu^*(K_X + \Delta).$$
        We call $(X',\Delta')$ a {\em crepant pair} of $(X,\Delta)$.
    \end{defi}

    We recall a well-known result on the resolution of sub-klt pairs.
    
    \begin{prop}\label{existence-of-very-log-resolution}{\rm (\cite[Proposition 1.10.7]{Kaw24})}
        Let $(X, \Delta)$ be a sub-klt sub-pair. Then there exists a birational morphism $\nu : Y \to X$ such that, if we uniquely write $$K_Y + \Gamma  - F = \nu^*(K_X + \Delta),$$
        where $\Gamma$ and $F$ are effective with no common components, and denote by $E$ the sum of all reduced exceptional divisors, then the following holds:
        \begin{itemize}
            \item[(1)] $X$ is smooth, $\Gamma + E$ has simple normal crossing support; 
            \item[(2)] $(Y,\Gamma)$ is terminal, and the components of $\Gamma$ pairwisely do not intersect;

            \item[(3)] if $K_X + \Delta$ is pseudo-effective, then  $\Gamma + E + N_\sigma(\nu^*(K_X + \Delta))$ also has simple normal crossing support.  
        \end{itemize}
    \end{prop}

    We need the definition for potentially birational divisors. We refer to \cite[Definition 2.3.3]{HMX13}.
    
    \begin{defi}\label{defi-for-p-birat}\em
        Let $X$ be a normal projective variety. Let $D$ be a $\bQ$-Cartier $\bQ$-divisor. We say $D$ is {\em potentially birational} if for any two different general points $x,y \in X$, after possibly switching $x$ and $y$, we can find a $\bQ$-divisor $0 \leq \Delta \sim_{\bQ}(1-\epsilon)D$ for some $\epsilon \in (0,1)$, such that $(X,\Delta)$ is lc but not klt at $x$, $x$ is an isolated lc center and $(X,\Delta)$ is not klt at $y$.
    \end{defi}

\subsection{B-divisors and Generalized Pairs} 
    
    We recall the definition of b-divisors introduced by Shokurov. For details, one may refer to \cite{FG14}.\par
    \begin{defi}\em
        Let $\mathbb{K}$ be $\bZ,\bQ$,  or $\bR$. Let $X$ be a normal projective variety. A {\em b-$\mathbb{K}$-divisor} $\mathbf{D}$ of $X$ consists of a series of $\mathbb{K}$-divisors $\{D_{X'}\}$ indexed by all the birational models $X'$ of $X$ which satisfies that, for any birational morphism $\nu : X'' \to X'$ between two birational models $X''$ and $X'$, $\nu_*D_{X''} = D_{X'}$ holds. For a b-$\bR$-divisor, we always omit the symbol $\bR$ and just call it a b-divisor. \par
        
        Given a b-$\mathbb{K}$-divisor $\mathbf{D}$ of $X$ and a birational model $Y$ of $X$, we call $D_{Y}$ the {\em trace} of $\mathbf{D}$ on $Y$ and we also denote it by $\mathbf{D}_Y$.
        
        We say that a b-$\mathbb{K}$-divisor $\mathbf{D}$ is {\em b-$\mathbb{K}$-Cartier} if there exists a birational morphism $f:X' \to X$ such that 
        \begin{itemize}
            \item[(1)]  $\mathbf{D}_{X'}$ is $\mathbb{K}$-Cartier, and
            \item[(2)]  for any birational morphism $\nu : X'' \to X'$, $\nu^*\mathbf{D}_{X'} = \mathbf{D}_{X''}$ holds. 
        \end{itemize}
        We say that $\mathbf{D}$ is {\em b-$\mathbb{K}$-Cartier-nef} if furthermore
        \begin{itemize}
            \item[(3)] $\mathbf{D}_{X'}$ is nef. 
        \end{itemize}
        In this case, we say the b-divisor $\mathbf{D}$ {\em descends} to $X'$.
    \end{defi}

    \begin{defi}\em 
        A {\em generalized sub-pair} $(X,\Delta + \mathbf{M}_X)$ consists of a normal quasi-projective variety $X$, an $\bR$-divisor $\Delta$ and a b-$\bR$-Cartier-nef b-divisor $\mathbf{M}$ consisting of $\{M_{X'}\}$ such that $K_X + \Delta + \mathbf{M}_{X}$ is $\bR$-Cartier. If additionally $\Delta \geq 0$, we call $(X,\Delta + \mathbf{M}_X)$ a {\em generalized pair}. \par
    \end{defi}

    \begin{defi}\em
        Let $(X,\Delta +\mathbf{M}_X)$ be a generalized sub-pair, and $D$ be a prime divisor over $X$. Let $f: Y \to X$ is a log resolution of $(X, \Delta + \mathbf{M}_X)$ such that $D$ is a divisor on $Y$ and $\mathbf{M}$ descends to $Y$.  we write uniquely
        $$K_Y + \Gamma + \mathbf{M}_Y = f^*(K_X + \Delta + \mathbf{M}_X).$$
        We define the {\em generalized log discrepancy} of the prime divisor $D$ with respect to the generalized sub-pair $(X,\Delta + \mathbf{M}_X)$ to be 
            $$a(D,X,\Delta + \mathbf{M}_X):= 1 - \mu_D(\Gamma).$$
        Fix a number $\epsilon>0$. We say that $(X,\Delta + \mathbf{M}_X)$ is {\em sub-glc} (resp. {\em sub-gklt, sub-$\epsilon$-glc}) if $a(D,X,\Delta + \mathbf{M}_X) \geq 0$ (resp. $>0, \geq \epsilon$) holds for any prime divisor $D$ over $X$. If, in addition, $(X,\Delta+ \mathbf{M}_X)$ is a generalized pair, we say that $(X,\Delta + \mathbf{M}_X)$ is {\em glc} (resp. {\em gklt, $\epsilon$-glc}). Similarly, we can define {\em glc centers}, {\em glc places} as in Definition \ref{Pair-and-sing}.
    \end{defi}
    
    We next recall the uniform lower bound for log canonical thresholds for generalized version established in \cite{Bir21b}.
        
    \begin{thm}\label{bounded-glct}
        {\rm (\cite[Lemma 2.9]{Bir21b})} Let $d \in \bN$, $r>0$, $\epsilon >0$. Then, there is a real number $t>0$ that depends only on $d,r,\epsilon$ that satisfies the following. Assume that 
        \begin{itemize}
            \item[(1)] $(X,\Delta + \mathbf{M}_X)$ is a projective $\epsilon$-glc pair of dimension $d$;
            \item[(2)] $A$ is a very ample divisor on $X$ with $A^d \leq r$;
            \item[(3)] $D$ is an effective $\bR$-divisor on $X$ and $\mathbf{N}$ is a b-$\bR$-Cartier-nef b-divisor such that $D + \mathbf{N}_X$ is $\bR$-Cartier; 
            \item[(4)] $A-(\Delta + \mathbf{M}_X + D + \mathbf{N}_X)$ is pseudo-effective.
        \end{itemize}
        Then, the generalized pair 
            $$(X, \Delta+tD + \mathbf{M}_X + t\mathbf{N}_X)$$
        is gklt.       
    \end{thm}

\subsection{Iitaka dimension and Iitaka volumes.}
    \begin{defi}\em
        Let $D$ be an $\bR$-divisor on a normal projective variety $X$. We define $\kappa_{\iota}(D)$, the {\em Iitaka dimension} of $D$ as follows.\par 
        If the $\bR$-linear system $|D|_{\bR}$ is empty, we set $\kappa(D) = -\infty$. Otherwise, we take an $\bR$-divisor $0 \leq D' \sim_{\bR}D$. For any non-negative integer $m$, the linear system $|\rounddown{mD'}|$ induces a rational map 
            $$\varphi_m : X \dashrightarrow \bP^{N_m},$$
        where $N_m = h^0(\rounddown{mD'}) - 1$. \par
        Then we define $$\kappa_{\iota}(D):= \max \{\dim \varphi_m(X) \mid m \in \bN\}.$$ 
        Note that this is independent of the choice of $D'$.
    \end{defi}

    \begin{defi}\em
        Let $D$ be an $\bR$-divisor on a normal projective variety $X$. If $\kappa_{\iota}(D) \geq 0$, we take a $\bR$-divisor $0 \leq D' \sim_{\bR} D$ and define the {\em Iitaka volume} of $D$ to be
        $$\ivol(D) := \limsup_{m \to +\infty} \frac{h^0(\rounddown{mD'})}{m^{\kappa_{\iota}(D)}/\kappa_{\iota}(D)!}. $$
        This is also independent of the choice of $D'$. If $\kappa_{\iota}(D) =-\infty$, we do not define the Iitaka volume of $D$. In particular, if $D$ is big, we define $\vol(D):=\ivol(D)$ and call it the {\em volume} of $D$.
    \end{defi}

\subsection{Stable minimal models.}\label{stable-minimal-models}
    We give the definitions of stable minimal models. For details we refer to \cite{Bir21b,Zhu25}.
    \begin{defi}\em (=Definition \ref{main-defi-for-smm})
        Let $d > 0$ be an integer, $\Phi \subseteq[0,1]$ be a set, $u,v>0$ be two real numbers. A {\em klt $(d,\Phi,u)$-stable minimal model} $(X,B),A$ consists of a projective klt pair $(X,B)$ and a divisor $A$ on $X$ such that
        \begin{itemize}
            \item[(1)] $\dim X = d$, and the coefficients of $B$ belong to $\Phi$; 
            \item[(2)] $K_X + B$ is semi-ample and gives a contraction $h: X \to Z$;
            \item[(3)] $A$ is an integral divisor that is big and semi-ample over $Z$ with $0 \leq \vol(A|_F) \leq u$ for a general fiber $F$ of $h:X \to Z$.
        \end{itemize}
    We say that $((X,B),A)$ is a {\em klt $(d,\Phi,{\leq}u,{\leq}v({\rm resp.} {=}v, {\geq}v))$-stable minimal model} if further
    \begin{itemize}
        \item[(4)] $\ivol(K_X + B) \leq v$ (resp. ${=}v, {\geq}v$) holds.
    \end{itemize}
    \end{defi}

    Note that in the above definition, we do not require $X$ to be $\bQ$-factorial. \par

    \begin{defi}\label{weak-smm}\em
        Let $d,\Phi,u,v$ be as above, and $(X,B)$ be a projective pair. We say that $(X,B)$ is a {\em weak klt $(d,\Phi, {\leq}u, {\leq}v({\rm resp. } {=}v, {\geq v}))$-stable minimal model} if there exists a klt $(d,\Phi, {\leq}u, {\leq}v({\rm resp. } {=}v, {\geq v}) )$-stable minimal model ($(X',B'),A'$), and a birational map $X \dashrightarrow X'$ such that $(X,B)$ and $(X', B')$ are crepant.      
    \end{defi}

    \begin{defi}\label{admits-klt-stable-minimal-model}\em
        Let $(X,B)$ be a projective klt pair and $d,\Phi,u,v$ be as above. We say that $(X,B)$ {\em admits a klt $(d,\Phi, {\leq}u, {\leq}v({\rm resp. } {=}v, {\geq v}))$-stable minimal model} if there exists a klt $(d,\Phi, {\leq}u, {\leq}v({\rm resp. } {=}v, {\geq v}))$-stable minimal model $((X',B'),A')$ such that $(X', B')$ is a weak lc model of $(X,B)$ (for the definition of weak lc model, see \cite[Definition 2.1]{Bir12}). 
    \end{defi}

\subsection{Nakayama-Zariski decomposition of divisors.}\label{NZ-decomp}

    We refer mainly to \cite[Chapters III and V]{N04} for the following results.

    \begin{defi}\em
        Let $X$ be a complex smooth projective variety. For a big $\bR$-divisor $D$ on $X$ and any prime divisor $\Gamma$, define 
        $$\sigma_{\Gamma}(D) := \inf\{\mu_{\Gamma}(D') \mid 0 \leq D' \sim_\bR D\}.$$
        If $D$ is only pseudo-effective, define $\sigma_{\Gamma}(D) := \lim_{\epsilon \to 0}\sigma_{\Gamma}(D + \epsilon A)$, where $A$ is any ample divisor. Note that this definition does not depend on the choice of $A$. \par 
        Then we define
            $$N_\sigma(D) := \sum_{\Gamma}\sigma_{\Gamma}(D) \Gamma \quad
            \text{and}\quad P_\sigma(D):= D -N_\sigma(D).$$
        We call the decomposition $D = P_\sigma(D) + N_\sigma(D)$ the {\em $\sigma$-decomposition}. If $P_\sigma(D)$ is nef, we say that $D$ {\em admits a Zariski-decomposition}.\par        
    \end{defi}
    Also, we need the relative version of Nakayama-Zariski decomposition of divisors.
    \begin{defi}\em 
        Let $f: X \to Y$ be a projective surjective morphism between smooth varieties and $D$ be an $f$-pseudo-effective $\bR$-divisor on $X$. Assume that $D$ is $f$-big and $\Gamma$ is a prime divisor. Then we define 
            $$\sigma_{\Gamma}(D;X/Y) := \inf\{\mu_{\Gamma}(D') \mid 0 \leq D' \sim_{\bR,f} D\}.$$

        If $D$ is only $f$-pseudo-effective, we define $$\sigma_{\Gamma}(D;X/Y):= \lim_{\epsilon \to 0}\sigma_{\Gamma}(D + \epsilon A;X/Y)$$ for a relatively ample divisor $A$. Again it is seen that the definition does not depend on the choice of $A$. If we have $\sigma_{\Gamma}(D;X/Y) < + \infty$ for any prime divisor $\Gamma$ on $X$, then we define that
            $$N_\sigma(D;X/Y) := \sum_{\Gamma}\sigma_{\Gamma}(D;X/Y)\Gamma,$$
        and that $P_\sigma(D;X/Y):= D -N_\sigma(D;X/Y)$.
    \end{defi}

    \begin{lemma}\label{property-for-NZ-decompose}{\rm (\cite[Theorem 5.15]{N04})}
        Let $f:X \to Y$ be a contraction between smooth varieties. Let $D_X$ be a prime divisor on $X$ and $D_Y$ be a prime divisor on $Y$ with $f(D_X) = D_Y$. Then for any pseudo-effective divisor $D$ on $Y$, we have 
        $$\sigma_{D_{X}}(f^*D) = (\mu_{D_X}f^*D_Y)\sigma_{D_{Y}}(D).$$
        In particular, we have 
        $$N_\sigma(f^*D) = f^*(N_\sigma(D)).$$
    \end{lemma}

    We need the following property of the Zariski negative part in a family.
    \begin{lemma}\label{Zariski-negative-in-family}
        Assume that $(V,C)$ is a projective lc pair and $f: V \to T$ is a contraction with a general fiber $X$. If $K_V + C$ is pseudo-effective, then 
        $$N_\sigma(K_V+C;V/T)|_X= N_\sigma(K_X + C_X).$$
    \end{lemma}

    \begin{proof}
        We take $\nu: V' \to V$ to be a log resolution of $(V,C)$ such that $g:=\nu|_{X'} : X' \to X$ is also a log resolution of $(X,C|_X)$, where $X'$ is the strict transform of $X$.
        $$
        \xymatrix{
            X' \ar@{^(->}[r]\ar[d]_{g} & V'\ar[d]^{\nu} & \\
            X  \ar@{^(->}[r]& V\ar[r]^{f}  & T
        }   
        $$
        We write uniquely 
        $$\nu^*(K_V + C) = K_{V'} + \Gamma- F,$$
        where $\Gamma,F$ are effective with no common components. Since $F$ is supported on the exceptional locus of $\nu$, we see that
        $$N_\sigma(K_{V'} + \Gamma; V'/T) = \nu^*(N_\sigma (K_V + C;V/T)) + F.$$
        Since $X$ is a general fiber, $X'$ only meets horizontal$/T$ part of $\Gamma$ and $F$. Hence
        $$N_{\sigma}(K_{X'} + \Gamma|_{X'}) = g^*(N_\sigma(K_X + C|_X)) + F|_{X'}.$$
        On the other hand, as $(V',\Gamma')$ is a log smooth lc pair, \cite[Lemma 2.3.2]{HMX18} gives 
        $$N_\sigma(K_{V'} + \Gamma; V'/T)|_{X'} = N_{\sigma}(K_{X'} + \Gamma|_{X'}).$$
        Combining these equalities yields
        $$N_\sigma(K_V+C;V/T)|_X= N_\sigma(K_X + C_X),$$
        completing the proof.
    \end{proof}
    
\subsection{Adjunction formula for fibrations.}{\label{adjunction}} 
    We recall the generalized adjunction formulas for fiber spaces, for which we mainly refer to \cite{Kaw98,Amb99,Amb04}.\par

    \begin{defi}\em
        Let $(X,B + \mathbf{M}_X)$ be a generalized sub-pair. Let $\mathbf{A}(X,B + \mathbf{M}_X)$ be a b-divisor that consists of $\{A_{X'}\}$ for all birational models $X'$ of $X$, where $A_{X'}$ are defined as follows. For any birational morphism $\nu : X' \to X$ from a normal variety $X'$, we define $$A_{X'}:=K_{X'}  + \mathbf{M}_{X'}- \nu^*(K_X + B + \mathbf{M}_X).$$
        We call $\mathbf{A}(X,B + \mathbf{M}_X)$ the {\em generalized discrepancy b-divisor } of $(X,B + \mathbf{M}_X)$. We define a sheaf on $X$ to be $\OO_{X}(\roundup{\mathbf{A}(X,B + \mathbf{M}_X)}) :=\tilde{\nu}_*\OO_{Y}(\roundup{A_{Y}})$ where $\tilde{\nu}: Y \to X$ is any log resolution of $(X,B + \mathbf{M}_X)$ with $\mathbf{M}$ descending to $Y$.

    \end{defi}

    \begin{defi}\em 
        A {\em generalized klt-trivial fibration} $f: (X,B + \mathbf{M}_X) \to Z$ consists of a contraction $f:X \to Z$ between normal projective varieties and a generalized sub-pair $(X,B + \mathbf{M}_X)$ that satisfies the following conditions:
        \begin{enumerate}
            \item[(1)] $(X,B + \mathbf{M}_X)$ is generalized sub-klt over the generic point of $Z$;
            \item[(2)] ${\rm rank}\ f_*\OO_{X}({\roundup{\mathbf{A}(X,B + \mathbf{M}_{X})}}) =1$;
            \item[(3)] $K_X + B + \mathbf{M}_X \sim_{\bR,f} 0$.
        \end{enumerate}
        In particular, if $\mathbf{M}_X = 0$, we simply say $f: (X,B) \to Z$ is a {\em klt-trivial fibration}.
    \end{defi}

    \begin{cons}\label{Adjuntion-for fiberspace}\em
        Given a generalized klt-trivial fibration $$f: (X,B + \mathbf{M}_X) \to Z,$$
        we will do adjunction to get a generalized pair on $Z$.\par

        For any prime divisor $P$ on $Z$, denote by $\eta_p$ the generic point of $P$. We define
        $$t_P := \sup\{a\in \bR \mid (X,B + af^*P + \mathbf{M}_X)\text{ is generalized sub-lc over $\eta_p$.}\}.$$\par 
        Next, we define $B_{Z} := \sum_{P}(1-t_P )P$, where the sum is taken over all prime divisors on $Z$. Then it is seen that $B_{Z}$ is a divisor and we call $B_{Z}$ the {\em discriminant part} of $(X,B +\mathbf{M}_X)$ on $Z$ (see \cite[Section 3.4]{FG14}). Then there exists an $\bR$-divisor $M_{Z}$ such that 
        $$K_{X} + B \sim_\bR f^*(K_{Z} + B_{Z} + M_{Z}).$$
        We call $M_{Z}$ the {\em moduli part} of $(X,B + \mathbf{M}_X)$ on $Z$. \par

        For any birational morphism $\sigma: Z' \to Z$, we take a birational morphism $\pi : X' \to X$ such that the induced rational map $f' : X' \dashrightarrow Z'$ is a morphism. Let 
        $$K_{X'} + B_{X'} + \mathbf{M}_{X'} = \pi^*( K_X + B_X + \mathbf{M}_X)$$
        be the crepant pull-back. We can similarly define the discriminant part $B_{Z'}$ and the moduli part $M_{Z'}$ on $Z'$. Although $M_{Z'}$ is only determined up to $\bR$-equivalence, we can take $M_{Z'}$ such that $M_{Z}$ is the pushdown of $M_{Z'}$. Thus, we can regard $\{M_{Z'}\}$ as a b-divisor $\overline{\mathbf{M}}$.
    
    \end{cons}

    \begin{lemma}\label{Existence-for-generalized}{\rm (See \cite[Theorem 0.2]{Amb04} and \cite[Theorem 2.23]{JLX22})}
        Let $f: (X,B + \mathbf{M}_X) \to Z$ be a generalized klt-trivial fibration. We construct the moduli b-divisor $\overline{\mathbf{M}}$ as above. Then $\overline{\mathbf{M}}$ is b-$\bR$-Cartier-nef and $(Z,B_Z + \overline{\mathbf{M}}_{Z})$ is a generalized sub-pair. 
    \end{lemma}

\subsection{Excellent models, strong excellent models for adjunction.}
    Next, we give the definition of excellent models and strong excellent models for a generalized klt-trivial fibration.

    \begin{defi}\label{defn-excellent-model}\em 
        Let $f:(X,B + \mathbf{M}_X) \to Z$ be a generalized klt-trivial fibration. Let $\sigma:Z' \to Z$, $\pi : X' \to X$ be birational morphisms where $X',Z'$ are projective and the induced map $f' : X' \dashrightarrow Z'$ is a morphism. We say that $f' :X' \to Z'$ is an {\em excellent model} with respect to $f:(X,B + \mathbf{M}_X) \to Z$ if it satisfies that
        \begin{itemize}
            \item[(1)] the moduli part $\mathbf{M}_{X'}$ and $\overline{\mathbf{M}}_{Z'}$ descend to $X'$ and $Z'$ respectively. 
        \end{itemize}
        We say that $f': X' \to Z'$ is a {\em strong excellent model} with respect to $f:(X,B + \mathbf{M}_X) \to Z$ if it further satisfies that
        \begin{itemize}
            \item[(2)] $(X',B')$ is log smooth and teminal, where $(X',B' + \mathbf{M}_{X'})$ is the crepant model of $(X,B+\mathbf{M}_X)$; and
            \item[(3)] for any prime divisor $P$ on $X'$ with $\mu_{P}B'>0$, the image $f'(P)$ is a prime divisor on $Z'$. 
        \end{itemize}
        
    \end{defi}

    \begin{lemma}\label{Existence-Strong-Excellent}
        Let $f:(X,B +\mathbf{M}_X) \to Z$ be a generalized klt-trivial fibration. Then we have the following.
        \begin{itemize}
            \item[(1)] There exists an excellent model with respect to $f:(X,B +  \mathbf{M}_X) \to Z$.
            \item[(2)] If  $(X,B + \mathbf{M}_X)$ is a sub-gklt generalized sub-pair, then there exists a strong excellent model with respect to $f:(X,B + \mathbf{M}_X) \to Z$.
        \end{itemize}
    \end{lemma}

    \begin{proof}
        The statement (1) follows by Lemma \ref{Existence-for-generalized}. \par
        Next, we prove the statement (2). By Proposition \ref{existence-of-very-log-resolution}, we can find a crepant model $(X',B' + \mathbf{M}_{X'})$ such that $(X',B')$ is log smooth and terminal so that there are only finitely many prime divisors over $X'$ with log discrepancy less than one. By \cite{H75} or \cite[Theorem 3.3]{Vil06}, we can take a birational model $\sigma:Z' \to Z$ that flattens these finitely many divisors. Finally, we replace $(X',B' + \mathbf{M}_{X'})$ with a higher log smooth model such that $f' : X' \dashrightarrow Z'$ is a morphism. Then, $f': X' \to Z'$ is a strong excellent model.
    \end{proof}
    
    Next, we recall a property about singularities while doing adjunction for fiber spaces. \par

    \begin{lemma}\label{klt-property-for-adjunction}

        Let $f:(X,B + \mathbf{M}_X) \to Z$ be a generalized klt-trivial fibration. We do adjunction to get a generalized sub-pair $(Z,B_Z + \overline{\mathbf{M}}_Z)$. Then, $(X,B + \mathbf{M}_X)$ is generalized sub-klt if $(Z,B_{Z} + \overline{\mathbf{M}}_{Z})$ is generalized sub-klt.
      
    \end{lemma}
    \begin{proof}
        The spirit of proof is similar to \cite[Theorem 2.11]{CL24}. 
        Suppose that $(X,B + \mathbf{M}_X)$ is not generalized sub-klt, then there is a prime divisor $D$ over $X$ such that $a(D,X,B +\mathbf{M}_X) \leq 0$. Take a birational model $X' \to X$ such that $D$ is a divisor on $X'$ and then take a birational model $Z' \to Z$ such that the image of $D$ on $Z'$ is a divisor and that $\overline{\mathbf{M}}$ descends to $Z'$. Finally, we replace $X'$ with a higher log smooth model so we can assume there is an induced morphism $f': X' \to Z'$ and that $\mathbf{M}$ descends to $X'$. \par
        Let $(X',B' + \mathbf{M}_{X'})$ be the crepant pullback of $(X, B + \mathbf{M}_X)$, and write 
        $$K_{X'} + B_{X'} + \mathbf{M}_{X'} \sim_{\bQ} f'^*(K_{Z'} + B_{Z'} + \overline{\mathbf{M}}_{Z'}).$$
        Since $a(D,X,B + \mathbf{M}_X) \leq 0$, $\mu_D(B_{X'}) \geq 1$. By the definition of discriminant part, we see that $\mu_{f'(D)}(B_{Z'}) \geq 1$, which means that $(Z,B_Z + \overline{\mathbf{M}}_Z)$ is not generalized sub-klt, a contradiction.
    \end{proof}

\subsection{Boundedness properties for stable minimal models.}

    We recall the boundedness property of the base of a stable minimal model.

    \begin{thm}\label{bounded-base-for-smm}
        Let $d \in \bN$, $\Phi \subseteq[0,1]$ be a DCC set, $u,v,\epsilon>0$ be real numbers. Then there exist positive numbers $r,\eta$ that depend only on $d,\Phi,u,v,\epsilon$ satisfying the following. Assume that 
        \begin{itemize}
            \item $(X,B),A$ is a klt $(d,\Phi,{\leq}u,{\leq}v)$-stable minimal model with $(X,B)$ being $\epsilon$-lc; 
            \item $h:X \to Z$ is the contraction determined by the semi-ample divisor $K_X + B$, and we write
            $$K_X + B \sim_{\bR} h^*(K_Z + B_Z + M_Z)$$
            and get a generalized pair $(Z,B_Z + \overline{\mathbf{M}}_Z)$. 
        \end{itemize}
        Then there exists a very ample divisor $A_Z$ on $Z$ with ${A_Z}^{\dim Z} \leq r$ such that $A_Z - B_Z$, $A_Z - M_Z$ are pseudo-effective and that $(Z,B_Z + \overline{\mathbf{M}}_Z)$ is $\eta$-glc.

    \end{thm}

    \begin{proof}
        The proof is very similar to the proof of \cite[Theorem 1.3]{BH24}. For the convenience of readers, we give a whole proof here. \par 
        By \cite[Theorem 1.9]{bir23b}, there is a positive number $\eta$ depending only on $d,u,\epsilon$ such that $(Z,B_Z + \overline{\mathbf{M}}_Z)$ is $\eta$-glc. By \cite[Lemma 3.3]{Zhu25}, there exists a finite set $I$, such that we can decompose
        $$\overline{\mathbf{M}}_{Z'} = \sum_{i=1}^n m_iM_{i,Z'},$$
        where $M_{i,Z'}$ are all nef Cartier divisors on a higher birational model $Z'$ and $m_i \in I$. By ACC for log canonical thresholds (see \cite[Theorem 1.1]{HMX14}), there is a DCC set $\Psi$ that depends on $d,\Phi$, such that $B_Z \in \Psi$.
        
        Also, we see that 
        $$\vol(K_Z + B_Z + \overline{\mathbf{M}}_Z) = \ivol(K_X + B) \leq v.$$
        Thus, $(Z,B_Z + \overline{\mathbf{M}}_Z)$ is log bounded by the proof of \cite[Lemma 6.8]{BH24}. \par 
        Therefore, there exists a uniform $r$, and a very ample divisor $A_Z$ on $Z$ with ${A_Z}^{\dim Z} \leq r$ such that $A_Z-K_Z$, $A_Z - B_Z$, and $A_Z - \overline{\mathbf{M}}_Z$ are all pseudo-effective.
    \end{proof}

    The next lemma concerns boundedness properties of singularities for stable minimal models with a fixed Iitaka volume.

    \begin{lemma}\label{DCC-fixed-volume-imply-epsilon-lc}{\rm (\cite[Theorem 1.6]{Zhu25})}
        Let $d > 0$ be an integer, $\Phi \subseteq[0,1]$ be a DCC set, $u,v> 0$ be real numbers. Then, there exists a real number $\epsilon > 0$, such that the following property holds. Assume that $(X,B),A$ is a klt $(d,\Phi,{\leq}u,{=}v)$-stable minimal model. Then, $(X,B)$ is $\epsilon$-lc.
    \end{lemma}

\subsection{Adjunction for lc centers.}\label{adjuntion-for-lc-centers}

    We recall some basic properties for pure and exceptional lc centers. \par
    \begin{lemma}\label{basic-for-adjunction-lc-centers}
        Let $(X,\Delta)$ be a pair with a pure and exceptional lc center $G \subseteq X$ with unique lc place $S$. Denote by $\nu : F \to G$ the normalization. Let $f : Y \to X$ be a log resolution of $(X,\Delta)$ such that $S$ is a divisor on $Y$. Write $K_Y + \Gamma = f^*(K_X + \Delta)$ and define $\Gamma_S := (\Gamma - S)|_S$. Denote by $g: S \to F$ the unique morphism induced by the restriction map of $f$ to $S$. Then, we have
        \begin{itemize}
            \item[(1)]  $g:(S, \Gamma_S) \to F$ is a klt-trivial fibration;
            \item[(2)] If we do adjuntion for $g:(S, \Gamma_S) \to F$ and write 
            $$K_S + \Gamma_S \sim_{\bR}g^*(K_F + B_F + M_F),$$
            where $B_F$ is the discriminant part and $M_F$ is the moduli part, then $B_F \geq 0$.
        \end{itemize}
    \end{lemma}

    \begin{proof}
        See \cite[Lemma 4.1, Lemma 4.2(3)]{Amb99} or the last part of \cite[Section 2.7]{CL24}.
    \end{proof}

\subsection{Koll\'ar's Injectivity Lemma.}
    We need a general version of the Kawamata-Viehweg vanishing theorem; to handle semi-ample but not big divisors we need to use Kollár’s Injectivity Lemma.
    \begin{thm}{\rm (\label{kollars-injectivity}\cite[Lemma 7.3]{KMM94})}
        Let $(V,\Delta)$ be a projective klt pair, and $L$ be an integral $\bQ$-Cartier divisor such that $L - (K_V + \Delta)$ is semi-ample. Assume that $D,E$ are effective $\bQ$-Cartier integral divisors with $D + E \in |m(L - (K_V + \Delta))|$ for some $m \in \bN$. \par 
        Denote by 
        $$\varphi_D: \mathcal{O}_V(L) \to \mathcal{O}_{V}(L + D)$$
        the injective map of sheaves induced by $D$. Then, 
        for any integer $i \geq 0$, the induced map on $i$-th cohomology
        $$H^i(V,\mathcal{O}_V(L) ) \to H^i(V,\mathcal{O}_V(L+D))$$
        is injective.  
    \end{thm}

    We need the following criterion for the separation properties of a semi-ample divisor.
    
    \begin{thm}\label{app-for-kol-inj}
        Let $(V,C)$ be a projective klt pair, and $L$ be an integral $\bQ$-Cartier divisor. Assume that 
        \begin{itemize}
            \item $K_V + C$ is semi-ample and induces a contraction $h:V \to Y$;
            \item there is a potentially birational ample $\bQ$-divisor $D$ on $Y$ such that 
            $$L - (K_V + C) \sim_{\bR} h^*D.$$
        \end{itemize}
        Then, $|L|$ separates different general fibers of $h$.
    \end{thm}

    \begin{proof}
        We do adjunction and write
        $$K_V + C \sim_{\bQ}h^*(K_Y + B_Y + M_Y)$$
        to get a generalized klt pair $(Y,B_Y + \mathbf{M}_Y)$ as in Section \ref{adjunction}.
        For two different general points $x,y \in Y$, denote by $F_x$, $F_y$ the corresponding fiber of $h$. Since $D$ is potentially birational and $x,y$ are general, there exists $\delta \in (0,1) \cap \bQ$ and $0 \leq D' \sim_{\bQ} (1 - \delta) D$, such that $x$ is an isolated glc center of $(Y,B_Y + D' + \mathbf{M}_Y)$ with unique glc place, while $(Y,B_Y + D' + \mathbf{M}_Y)$ is not gklt at $y$. Taking $\Delta=h^*D'$, we have $0 \leq \Delta \sim_{\bQ}(1-\delta)h^*D$ and $F_x$ is a pure and exceptional lc center of $(V,C + \Delta)$ with unique lc place $E$, and $(V, C + \Delta)$ is not klt at $F_y$. \par
        Let $\nu: W \to V$ be a log resolution of $(V,C)$ such that $E$ is a divisor on $W$. Write uniquely
        $$\nu^*(K_V + C + \Delta) = K_W + \Gamma - N,$$
        where $\Gamma,N$ are effective with no common components. Let $$M:= \roundup{\nu^*L - \Gamma + N}, \quad G:=\rounddown{\Gamma}.$$ We see that $M \geq \rounddown{\nu^*L} - G$ and that $M + G - \rounddown{\nu^*L}$ is $\nu$-exceptional. Also, we see that 
        \begin{align*}
            M  - (K_W + M - (\nu^*L-\Gamma + N)) &= \nu^*L - (K_W + \Gamma - N) \\
            &=\nu^*h^*(\delta D)
        \end{align*}
        is semi-ample and $(W, M - (\nu^*L-\Gamma + N))$ is klt. Since any support of $G$ is vertical over $Y$ and $D$ is an ample $\bQ$-divisor on $Y$, there exists some integer $m>0$ and another $\bQ$-divisor $G'\geq 0$, such that $G+G' \in |m(\nu^*h^*(\delta D))|$. By Theorem \ref{kollars-injectivity}, we get an injective map of cohomology
        $$H^1(W,M) \to H^1(W,M + G).$$
        As a result, we get a surjective map of global sections 
        $$H^0(W,\OO_W(M + G)) \to H^0(W,\OO_G((M + G)|_G)).$$
        Since $M + G \geq \rounddown{\nu^*L}$ and $M + G - \rounddown{\nu^*L}$ are $\nu$-exceptional, the linear systems $|\rounddown{\nu^*L}|$ and $|M + G|$ are isomorphic. As $F_x$ is a pure and exceptional lc center, every coefficient of $(\Gamma - E)|_E$ is less than $1$, and it follows that $(M+G)|_E$ is an effective divisor. Moreover, since $(V,C + \Delta)$ is not klt at $F_y$, the image of some component of $G$ contains $F_y$.\par
        Combining the above, there exists a global section $\sigma \in H^0(V,L)$ that does not vanish at a general point of $F_x$, but vanishes identically on $F_y$, as required.
    \end{proof}
    
\section{Stable minimal models in families.}\label{Stable-mm-in-famlies}

In this section, we will explore a special type of fibrations, of which a general fiber admits a klt stable minimal model.  

\subsection{Stable-minimal-model type fibrations.}\label{Special-fibrations}
    \begin{defi}\em
        Let $d > 0$ be an integer, $\Phi\subseteq[0,1]$, $u,v> 0$ be real numbers. Assume that $(V,C)$ is a projective klt pair and $f : V \to T$ is a contraction. We say that $f :(V,C) \to T$  is a {\em klt $(d,\Phi,{\leq}u,{\leq}v({\rm resp.} {=}v, {\geq}v))$-SMM-type fibration } if for any general fiber $X$, there is a divisor $A_X$ on $X$ such that $((X,C|_X),A_X)$ is a klt $(d,\Phi,{\leq}u,{\leq}v({\rm resp.} {=}v, {\geq}v))$-stable minimal model.  \par
        We say that $f :(V,C) \to T$  is a {\em weak klt $(d,\Phi,{\leq}u,{\leq}v({\rm resp.} {=}v, {\geq}v))$-SMM-type fibration} if for a general fiber $X$, $(X,C|_X)$ is a weak klt $(d,\Phi,{\leq}u,{\leq}v({\rm resp.} {=}v, {\geq}v))$-stable minimal model. (See Definition \ref{weak-smm})\par 
        We say that $f:(V,C) \to T$ {\em admits a klt $(d,\Phi,{\leq}u,{\leq}v({\rm resp.} {=}v, {\geq}v))$-SMM-type fibration}, if for a general fiber $X$, $(X,C|_X)$ admits a klt $(d,\Phi,{\leq}u,{\leq}v({\rm resp.} {=}v, {\geq}v))$-stable minimnal model. (See Definition \ref{admits-klt-stable-minimal-model})
           
    \end{defi}

    \begin{lemma}\label{MMP-inv-for-SMM-type}
        Let $d,\Phi,u,v$ be as above. Assume that $(V,C)$ is a klt projective pair and $f:V \to T$ is a contraction. Then, the following statements are equivalent:
        \begin{itemize}
            \item[(1)] $f:(V,C) \to T$ admits a klt $(d,\Phi,{\leq}u,{\leq}v)$-SMM-type fibration.
            \item[(2)] We can run MMP$/T$ on $K_V + C$ until we get a model $(V',C')$ such that $(V',C') \to T$ is a weak klt $(d,\Phi,{\leq}u,{\leq}v)$-SMM-type fibration.
        \end{itemize}
    \end{lemma}

    \begin{proof}
        First, we prove (2) implies (1). Assuming (2), take $p:V'' \to V$ and $q:V'' \to V'$ be a common log resolution of $(V,C)$ and $(V',C')$. By the properties of the MMP, we can write
        $$p^*(K_V + C) = q^*(K_{V'} + C') + E,$$
        where $E \geq 0$ is $q$-exceptional. Let $X$ be a general fiber of $f : V \to T$ and $X',X''$ be its strict transform on $V',V''$, respectively. The generality of $X$ implies that no component of $E$ contains $X''$. Restricting the above equality to $X''$ yields
        $$(p|_{X''})^*(K_{X} + C|_X) \geq (q|_{X''})^*(K_{X'} + C'|_{X'}).$$
        Since $(V',C') \to T$ is a weak klt $(d,\Phi,{\leq}u,{\leq}v)$-SMM-type fibration, $(X',C'|_{X'})$ is a weak klt $(d,\Phi,{\leq}u,{\leq}v)$-stable minimal model, so that $K_{X'} + C'|_{X'}$ is semi-ample. Hence, $(X',C'|_{X'})$ is a weak lc model of $(X,C|_X)$. This completes the first part of the proof.
        
        \bigskip
        
        Next, we prove that (1) implies (2). 
        Fix a general fiber $X$ of $f: V \to T$. By assumption, $(X,C|_X)$ has a semi-ample model, so it has a good minimal model by \cite[Lemma 2.9.1]{HMX18}. By the proof of \cite[Lemma 3.2]{HMX18}, we can run MMP$/T$ on $K_V + C$ with scaling until we get a model $(V',C')$ such that $K_{V'} + C'$ is nef over the generic point of $T$. Therefore, $K_{X'} + C'|_{X'}$ is nef for any general fiber $X'$ of $V' \to T$. Since $(X',C'|_{X'})$ admits a klt $(d,\Phi,{\leq}u,{\leq}v)$-stable minimal model, the negativity lemma implies that $(X', C'|_{X'})$ is a weak klt $(d,\Phi,{\leq}u,{\leq}v)$-stable minimal model. Hence, $(V',C') \to T$ is a weak klt $(d,\Phi,{\leq}u,{\leq}v)$-SMM-type fibration.
    \end{proof}

\subsection{Conditions $\mathcal{M}_1(t)$ and $\mathcal{M}_2(t)$.}

    We will need the conditions $\mathcal{M}_1(t)$ and $\mathcal{M}_2(t)$ defined below for a fibration, where $t>0$ is a real number. $\mathcal{M}_2(t)$ is strictly stronger than $\mathcal{M}_1(t)$.
    \begin{defi}\label{defi-for-C(t)}\em
        Given a projective klt $\bQ$-pair $(V,C)$ and a fibration $f: V \to T$. We say that $f :(V,C) \to T$ satisties {\em condition $\mathcal{M}_1(t)$} or is {\em of $\mathcal{M}_1(t)$-type} if for a general fiber $X$ of $f: V \to T$, there exists a $\bQ$-divisor $0 \leq \Delta \sim_\bQ \delta(K_V + C)$ for some $\delta \in (0,t)\cap \bQ$, such that $X$ is a maximal non-klt center of $(V, C +\Delta)$. 
    \end{defi}

    \begin{defi}\label{defi-for-M(t)}\em
            Given a projective klt $\bQ$-pair $(V,C)$ and a fibration $f: V \to T$. We say that $f :(V,C) \to T$ satisties {\em condition $\mathcal{M}_2(t)$} or is {\em of $\mathcal{M}_2(t)$-type} if the following holds. Assume that $X_1$, $X_2$ are two different general fibers of $f: V \to T$, then after possibly switching $X_1$ and $X_2$, we have
            \begin{itemize}
                \item[(1)] there exists a $\bQ$-divisor $0 \leq \Delta_1 \sim_{\bQ} \delta_1(K_V + C)$ for some $\delta_1 \in (0,t)\cap \bQ$, such that $X_1$ is a pure and exceptional lc center of $(V, C +\Delta_1)$ with unique lc place $E_1$ over $X_1$; and
                \item[(2)] there exists a $\bQ$-divisor $0 \leq \Delta \sim_{\bQ} \delta(K_V + C)$ for some $\delta \in (0,t) \cap \bQ$, such that $X_2$ is a pure and exceptional lc center of $(V,C + \Delta)$ and $a(E_1,V,C + \Delta) \leq 0$ holds.
            \end{itemize}
    \end{defi}

    Next, we list some properties for $\mathcal{M}_1(t)$-type and $\mathcal{M}_2(t)$ fibrations.

    \begin{lemma}\label{MMP-inv-for-C(t)}
        Let $(V,C)$ be a projective klt $\bQ$-pair, $f:V\to T$ a contraction, and $t>0$ a real number.  
        Let $(V',C')$ be another projective klt pair birational to $(V,C)$ over $T$, and let $p:V''\to V$, $q:V''\to V'$ be a common resolution.  
        Assume that the $\mathbb Q$-linear systems $|p^*(K_V+C)|_{\mathbb Q}$ and $|q^*(K_{V'}+C')|_{\mathbb Q}$ are isomorphic, then  
        $f:(V,C)\to T$ satisfies $\mathcal{M}_1(t)$ {\rm(resp. $\mathcal{M}_2(t)$)} if and only if $f':(V',C')\to T$ satisfies $\mathcal{M}_1(t)$ {\rm(resp. $\mathcal{M}_2(t)$)}.
        
        In particular, the conclusion holds when $(V',C')$ is a log minimal model of $(V,C)$ over $T$.
    \end{lemma}

    \begin{proof}
        Here we only give a proof for $\mathcal{M}_2(t)$ and leave the $\mathcal{M}_1(t)$ case for interested readers. \par
        By assumption we can write
        $$p^*(K_V + C) + F = q^*(K_{V'} + C') + E,$$
        where $E,F \geq 0$ with no common components. \par
        We only need to prove for one direction. Assume that $f:(V,C) \to T$ satisfies $\mathcal{M}_2(t)$. For different general fibers $X_1',X_2'$ of $f'$, let $X_1,X_2$ be their strict transform on $V$, respectively. By assumption, after possibly switching $X_1,X_2$, there exist a $\bQ$-divisor $0 \leq \Delta_1 \sim_{\bQ} \delta_1(K_V + C)$ with $\delta_1 \in (0,t) \cap \bQ$ that satisfies the condition (1) of Definition \ref{defi-for-M(t)}, and a $0 \leq \Delta \sim_{\bQ} \delta(K_V + C)$ with $\delta \in (0,t) \cap \bQ$ that satisfies the condition (2) of Definition \ref{defi-for-M(t)}. \par
        Take $$\Delta_1' = q_*p^*(\Delta_1+ F-E),   \text{ and } \Delta' = q_*p^*(\Delta+ F-E).$$
        
        Since $|p^*(K_V+C)|_{\mathbb Q}$ and $|q^*(K_{V'}+C')|_{\mathbb Q}$ are isomorphic, we see $ 0 \leq \Delta_1' \sim_{\bQ} \delta_1(K_{V'} + C')$ and $ 0 \leq \Delta' \sim_{\bQ} \delta(K_{V'} + C')$. Since $X_1',X_2'$ are general fibers so that not contained in the image of $E$ and $F$, for any prime divisor $D$ over $V'$ lying over $X_1'$ or $X_2'$, we have equality of log discrepancies
        \begin{align*}
            a(D,V,C+\Delta_1) &= a(D,V',C'+\Delta_1') \\
            a(D,V,C+\Delta) &=a(D,V',C'+\Delta'),
        \end{align*}
        which implies $f':(V',C') \to T$ satisfies $\mathcal{M}_2(t)$.
    \end{proof}

    \begin{lemma}\label{Iitaka-property-for-C(t)}
        Let $(V,C)$ be a projective klt $\bQ$-pair, $f:(V,C) \to T$ be a contraction that satisfies conditon $\mathcal{M}_1(t)$, where $t>0$ is a real number. Assume that $(V_0,C_0)$ is a good minimal model of $(V,C)$ and $K_{V_0} + C_0$ gives a contraction $h_0: V_0 \to Y$. Then, $h_0$ is generically over $T$. In other words, there exist birational morphisms $V' \to V$ and $Y' \to Y$ such that the induced rational map $h':V' \dashrightarrow Y'$ is a morphism, and there is a morphism $g : Y' \to T$ such that $f' = g \circ h'$, where $f':V' \to T$ is the induced contraction. (see the following diagram.)
        $$
        \xymatrix{
            V' \ar[rr]^{h'} \ar[dr]^{f'} \ar[dd]&     &   Y' \ar[dl]_{g} \ar[dd]\\
                & T &   \\
            V\ar[ur]^{f} \ar@{-->}[r] \ar@/_1.5pc/[rr]_{h}  &  V_0\ar[r]^{h_0}  &   Y \ar@{-->}[ul]  \\
        }   
        $$
    \end{lemma}

    \begin{proof}
        First, after replacing $(V,C)$ with a crepant pair, we can assume $q:V \to V_0$ is a morphism. By negativity lemma, we can write
        $$K_V + C = q^*(K_{V_0} + C_0) + E$$
        for some $q$-exceptional divisor $E \geq 0$. 
        As in Construction \ref{Adjuntion-for fiberspace}, we can do adjunction to uniquely write 
        $$K_{V_0} + C_0 \sim_{\bR}h_0^*(K_Y +B_Y + M_Y)$$
        to get a generalized klt pair $(Y, B_Y + \mathbf{M}_{Y})$. Let $h:  V \to Y$ be the induced morphism.

        Assume that any general fiber of $h$ is not contracted by $f$. Take $S$ to be a general fiber of $h$ and take a general point $x \in S$, and let $X$ be the unique fiber of $f $ passing through $x$. By assumption, there is a $0\leq \Delta \sim_{\bQ}\delta(K_V + C)$ for some $\delta \in (0,t)\cap \bQ$, such that $X$ is a maximal non-klt center of $(V,C + \Delta)$.  Since $h$ is the Iitaka fibration for $K_V + C$ and $E$ is numerically fixed in the numerical linear system $|K_V + C|_{\rm num}$, we see $\Delta = h^*D + \delta E$ for some $D \geq 0$ on $Y$. Since $X$ is a general fiber, $E$ makes no contribution to the log discrepancy of divisors lying over $X$. By Lemma \ref{klt-property-for-adjunction}, $(Y,B_Y + D + \mathbf{M}_Y)$ is not gklt and there is a generalized not-klt center $Z$ on $Y$ that contains $h(X)$. We see that $h^{-1}(Z)$ is a non-klt center of $(V,C+\Delta)$. Since $S\subseteq h^{-1}(Z)$ and $S$ is not contracted by $f$, we have $X\subsetneq h^{-1}(Z)$, contradicting the maximality of $X$ as a non-klt center.
    \end{proof}

    By the following lemma, after taking a suitable scaling and small perturbation we can turn any maximal non-klt center into a pure and exceptional lc center.

    \begin{lemma}\label{perturb} 
        Let $(V,C)$ be a projective klt $\bQ$-pair that has a good minimal model. $f: (V,C) \to T$ is a contraction that satisfies condition $\mathcal{M}_1(t)$. Then, for a general fiber $X$, there exists $\delta' \in (0,t) \cap \bQ$ and a $\bQ$-divisor $0 \leq \Delta' \sim_\bQ \delta'(K_V + C)$, such that $X$ is a pure and exceptional lc center of $(V, C +\Delta')$.
    \end{lemma}

    \begin{proof}
        Let $(V_0,C_0)$ be a good minimal model of $(V,C)$, and $h_0:V_0 \to Y$ be the contraction defined by semi-ample divisor $K_{V_0} + C_0$. By Lemma \ref{MMP-inv-for-C(t)}, after replacing $(V,C)$ with a crepant pair, we can assume that the rational map $q:V \dashrightarrow V_0$ is a morphism. We define $h:V \to Y$ to be the induced morphism.
        We do adjunction and use negativity lemma to write
        $$ K_V + C\sim_{\bQ} h^*(K_Y + B_Y + M_Y) + E,$$
        to get a generalized klt pair $(Y, B_Y + \mathbf{M}_Y)$, where $E \geq 0$ is $q$-exceptional. 

        By condition $\mathcal{M}_1(t)$, there exists $0\leq \Delta \sim_{\bQ}\delta(K_V + C)$ for some $\delta \in (0,t) \cap \bQ$ such that $X$ is a maximal non-klt center of $(V,C + \Delta)$. After replacing $\Delta$ with $\lct_{\eta_X}(V,C;\Delta){\cdot} \Delta$, where $\eta_X$ is the generic point of $X$, we can assume that $X$ is a maximal lc center of $(V,C + \Delta)$. Since $h: V \to Y$ is the Iitaka fibration for $K_V + C$, there is an ample $\bQ$-divisor $0 \leq D \sim_{\bQ} \delta(K_Y + B_Y + M_Y)$ on $Y$ with $\Delta = h^*D + \delta E$. By Lemma \ref{Iitaka-property-for-C(t)}, we can assume that $Z:=h(X)$ is a generalized lc center of $(Y,B_Y + D + \mathbf{M}_Y)$, and $X=h^{-1}(Z)$. Since $D$ is big, as in \cite[Lemma 2.32]{Bir19} or \cite[Lemma 2.2]{CL24}, we can perturb $D$ to get another $ 0 \leq D'\sim_{\bQ} \delta'(K_Y + B_Y +M_Y)$ for some $\delta' \in (\delta,t )\cap \bQ$, such that $Z$ is a maximal glc center of $(Y, B_Y + D' + \mathbf{M}_Y)$ with a unique lc place lying over $Z$. Take $\Delta':=h^*D'+ \delta'E$. We see that $\Delta' \sim_{\bQ} \delta'(K_V + C)$ and that $X$ is a pure and exceptional lc center of $(V, C +\Delta')$. 
    \end{proof}

\subsection{Birational families of intrinsic tigers.}

Here, we recall definition of birational families of intrinsic tigers in \cite[Section 3]{Lac23}. We also refer to \cite[Section 3]{Wang25}.

    \begin{defi}\em
        Let $X$ be a smooth projective variety of dimension $n$, $D$ be a $\bQ$-divisor on $X$. For any integer $k > 0$, and a point $x \in X$, we define 
        $$S^k_{x}(D):= \{H \in |\rounddown{D}| \mid \mult_x{H} \geq k.\}.$$
        We define 
        $${B'}^k_x(D):= \bigcap_{H \in S^k_{x}(D)} H.$$
        and define $B^k_{x}(D)$ to be the union of irreducible components of ${B'}^k_x(D)$ passing through $x$. We know that there is a Zariski open set $U \subseteq X$ such that for any $x \in U$, the dimension of $B^k_{x}(D)$ is constant. We denote this dimension by $d(D,k)$.
    \end{defi}
    
    \begin{defi}\label{defi-for-tigers}\em 
        Let $X$ be a projective variety of dimension $n$. Let $D$ be a $\bQ$-divisor, $w>0$ be a rational number. We say that $f : Y \to B$ is a {\em birational family of intrinsic tiger of weight $w$} relative to $D$ if there are integers $k,l>0$ and an open set $U \subseteq Y$, and a birational morphism $\pi: Y \to X$ such that
        \begin{itemize}
            \item[(1)] $d(lD,k) <n$; 
            \item[(2)] $\pi(f^{-1}(f(y))) = B^k_{\pi(y)}(lD)$, for $y \in U$;
            \item[(3)] Every element in $S^k_{\pi(y)}(lD)$ has muliplicity at least $wl(n - d(lD,k))$ along $B^k_{\pi(y)}(lD)$. 
        \end{itemize}
    \end{defi}

    The next Lemma provides the relationship between birational families of intrinsic tigers and the condition $\mathcal{M}_2(t)$. 
    
    \begin{lemma}\label{property-for-tigers}
        Let $(V,C)$ be a projective log smooth $\bQ$-pair. Assume that $f: V \to T$ is a birational family of intrinsic tiger of weight $w > 0 $ relative to $D:= K_V + C$, then $f:(V,C) \to T$ satisfies condition $\mathcal{M}_2(t)$, where $t:= 3/w$.
    \end{lemma}

    \begin{proof}
        Take $X_1$, $X_2$ to be different general fiber of $f$. By \cite[Lemma 3.4]{Lac23}, we see that $f:(V,C) \to T$ satisfies condition $\mathcal{M}_1(1/w)$. We divide into two steps to prove it satisfies condition $\mathcal{M}_2(t)$ for $t=3/w$.
        \bigskip
        
        {\em Step 1.} In this step, we prove that there exists $0 \leq \Delta' \sim_{\bQ} \delta' D$ for some $\delta' \in (0,2/w) \cap \bQ$, such that after possibly switching $X_1$ and $X_2$, $X_1$ is a pure and exceptional lc center of $(V,C +\Delta')$ and it is klt at the generic point of $X_2$.\par
        \bigskip
        For $i=1,2$, take a general $x_i\in X_i$. By \cite[Lemma 3.4]{Lac23}, there exist general $0 \leq \Delta_i \sim_{\bQ} \delta_i D$ for some $\delta_i \in (0,1/w] \cap \bQ$, such that $X_i$ is a pure lc center of $(V,\Delta_i)$. Here we make the following remark: because of the definition of birational family of intrinsic tigers, $X_i$ is exactly the base locus of $S^k_{x_i}(lD)$ for some integers $k,l>0$, so the linear system $S^k_{x_i}(lD)$ is base-point free in a punctured neighborhood of the generic point of $X_i$, for $i=1,2$, thus we can take $\Delta_i$ as above. \par
        Let $\Delta:= \Delta_1 + \Delta_2$ and then define 
        $$c:= \min \{\lct_{\eta_{X_1}}(V,C;\Delta),  \lct_{\eta_{X_2}}(V,C;\Delta)\}.$$
        After possibly switching $X_1$ and $X_2$, we can assume  $X_1$ is a pure lc center of $(V,C + c\Delta)$. Note that $c \in(0, 1) \cap \bQ$, so that $c\Delta \sim_{\bQ} \delta D$ for some $\delta \in (0, 2/w) \cap \bQ$. We will get into one of the following two cases: 
        
        {\bf Case I.} In this case, $(V,C+ c\Delta)$ is klt at the generic point of $X_2$. By Lemma \ref{perturb}, we can take $\Delta'$ to be a small perturbation of $c\Delta$ such that $X_1$ is a pure and exceptional lc center of $(V,C + \Delta')$ and keeps klt at the generic point of $X_2$. Thus, we finish this step. 
        
        {\bf Case II.} In this case, $(V,C + c\Delta)$ is lc but not klt at the generic point of $X_2$. Again by Lemma \ref{perturb}, we take $\Delta'$ to be a sufficiently general small perturbation of $c\Delta$ such that $X_1$ is a pure and exceptional lc center of $(V,C + \Delta')$. If $(V,C + \Delta')$ is klt at the generic point of $X_2$, we are done; else if $(V,C + \Delta')$ is not lc at the generic point of $X_2$, then $\lct_{\eta_{X_2}}(V,C;\Delta')<1$. Replacing $\Delta'$ with $\lct_{\eta_{X_2}}(V,C;\Delta')\cdot \Delta'$ and switching $X_1$ and $X_2$, we get to {Case I} again; else, we get $X_2$ is also a pure and exceptional lc center of $(V,C + \Delta')$ since $\Delta'$ is a general perturbation. But in this case, $f:(V,C) \to T$ has already satisfied condition $\mathcal{M}_2(t)$. Therefore, we finish this step. \par
        \bigskip
        
        {\em Step 2.} In this step, we finish the proof to show that $f:(V,C) \to T$ satisfies condition $\mathcal{M}_2(t)$. \par

        By Step 1, we denote the unique lc place of $(V,C + \Delta')$ lying over $X_1$ by $E_1$. Similarly, since the linear system $S^k_{x_2}(lD)$ is base-point free in a punctured neighborhood of the generic point of $X_2$, we can take general $0 \leq \Delta'' \sim_{\bQ} \delta'' D$ for some $\delta'' \in (0,1/w)\cap \bQ$, such that $X_2$ is a pure lc center of $(V, C + \Delta' + \Delta'')$. \par
        
        By Lemma \ref{perturb}, replacing $\Delta''$ with its small perturbation, we can assume that $X_2$ is a pure and exceptional lc center of $(V,C + \Delta' + \Delta'')$ and $a(E_1,V,C+\Delta'+\Delta'') \leq 0$. We define $\Delta''' = \Delta' + \Delta''$ and we see $\Delta''' \sim_{\bQ} \delta''' D$ for some $\delta''' \in (0,3/w)$. Therefore, we get $f:(V,C) \to T$ satisfies the condition $\mathcal{M}_2(t)$ for $t=3/w$.   
    \end{proof}

    Next theorem constructs a fibration structure for pairs with large Iitaka volumes. The idea is firstly proposed in \cite[Theorem 6.8]{CJ17} and improved in \cite{Lac23} and \cite{Wang25}. Here, we provide a stronger version. 
    \begin{thm}{\label{construct-fibration}}
        Let $n \geq d > 1$ be integers. Let $\lambda: \bZ_{>0}  \times \bR_{>0} \to \bR_{>0}$ be a fixed function. Then, there exist $v_{d-1} \geq \cdots \geq v_1 > 0$ and $v > 0$ such that the following holds. Assume that
        \begin{itemize}
            \item[(1)] $(V,C)$ is a projective klt $\bQ$-pair that has a good minimal model;
            \item[(2)] $K_V + C$ has Iitaka dimension $d$ and $\ivol(K_V + C) > v$.
        \end{itemize}
        Then, we will get either
        \begin{itemize}
            \item $|m'( K_V + C)|$ induces an Iitaka fibration, where $m'=\max\{m_0,2\}$ and $m_0>0$ is the minimal integer such that $m_0C$ is integral; or 
            \item  after replacing $(V,C)$ with a crepant pair, there is an integer $k \in(0,d)$ and a fibration $f:(V,C) \to T$ such that for a general fiber $X$, $\dim X = n - d + k$, and $f: (V,C) \to T$ satisfies condition $\mathcal{M}_2(\lambda(k,v_k))$ and $\ivol(K_X + C|_X) \leq v_k$.
        \end{itemize}
    \end{thm}

    \begin{proof}
        The proof is very similar to \cite[Lemma 3.8]{Wang25}. For the convenience of readers, here we give a complete proof. \par
        We replace $(V,C)$ with its good minimal model. Let $h : V \to Y$ be the contraction defined by semi-ample divisor $K_V + C$ and we can write $K_V + C \sim_{\bQ} h^*D$ for some ample $\bQ$-divisor $D>0$ on $Y$ such that $\vol(D) = \ivol(K_V + C)$. By the proof of \cite[Theorem 1.2]{Wang25} and \cite[Lemma 3.6]{Lac23}, there exist $v_{d-1} \geq \cdots \geq v_1 > 0$ and $v > 0$ such that either $D$ is potentially  birational, or after replacing $Y$ with a higher model, We will get a fibration $g: Y \to T$, such that
        \begin{itemize}
            \item[(1)] for a general fiber $F$ of $g: Y \to T$, $\dim F = k$ for some $1\leq k\leq d-1$ and $\vol(D|_F) \leq v_k$; and 
            \item[(2)] $g: Y \to T$ is a birational family of intrinsic tigers of weight $3/t_k$ for some $0<t_k \leq \lambda(k,v_k)$. 
        \end{itemize}
        
        In the first case when $D$ is potentially birational, then $|m'(K_V + C)|$ gives an Iitaka fibration by Theorem \ref{app-for-kol-inj}. \par
        
        In the second case, replacing $V$ with a higher smooth model, we can assume $h:V \to Y$ is a morphism and define $f:= g \circ h: V \to T$ be the composition. We see that $f : V \to T$ is a birational family of intrinsic tigers of weight $3/t_k$ relative to $K_V + C$. Let $X$ be a general fiber of $f: V \to T$. Then, $\ivol(K_X + C|_X) = \vol(D|_F) \leq v_k$. By Lemma \ref{property-for-tigers}, we know $f:(V,C) \to T$ satisfies condition $\mathcal{M}_2(t_k)$. We finish the proof. 
    \end{proof}

\section{Restriction for Negative parts}\label{Res-for-N-parts}

The goal of this section is to prove Theorem \ref{main-Res-for-Negative-parts} on the behavior of the restriction of the Zariski-Negative part to a general fiber. It will be used to prove the extension theorem in Section 5.

\begin{thm}\label{Res-for-Negative-parts}{\rm (= Theorem \ref{main-Res-for-Negative-parts})}
    Let $d>0$ be an integer, $\Phi \subseteq[0,1]$ be a DCC set, $ \epsilon, u,v>0$ be real numbers. Then, there exists a real number $t>0$ that satisfies the following. Assume that 
    \begin{itemize}
        \item[(1)] $f:(V,C) \to T$ admits a klt $(d,\Phi,{\leq}u, {\leq}v)$-SMM-type fibration such that $(X,C|_X)$ is $\epsilon$-lc for a general fiber $X$; 
        \item[(2)] $f:(V,C) \to T$ satisfies the condition $\mathcal{M}_2(t)$.  
    \end{itemize}
    Then, for a general fiber $X$, we have 
    $$N_\sigma(K_V + C)|_X = N_
    \sigma(K_X + C|_X).$$
\end{thm}

If we strengthen the condition ``$(d,\Phi,{\leq}u, {\leq}v)$" to ``$(d,\Phi,{\leq}u, {=}v)$", but weaken the condition``$\epsilon$-lc" to ``klt" in the above theorem, we will get following result.

\begin{cor}\label{Res-for Negative-parts-2}
        Let $d>0$ be an integer, $\Phi \subseteq[0,1]$ be a DCC set, $u,v>0$ be real numbers. Then, there exists a real number $t>0$ that satisfies the following. Assume that 
    \begin{itemize}
        \item[(1)]  $f:(V,C) \to T$ admits a klt $(d,\Phi,{\leq}u, {=}v)$-SMM-type fibration; 
        \item[(2)] $f:(V,C) \to T$ satisfies the condition $\mathcal{M}_2(t)$. 
    \end{itemize}
    Then, for a general fiber $X$, we have 
    $$N_\sigma(K_V + C)|_X = N_
    \sigma(K_X + C|_X).$$
\end{cor}

Before proving the theorem above, we establish a uniform boundedness result for singularities near a maximal lc centre. 

\begin{thm}\label{lc-property}
    Let $d > 0$ be an integer, $\Phi \subseteq[0,1]$ be a DCC set, $u,v,\epsilon>0$ be real numbers. Then, there exists a positive number $t$ that satisfies the following. Assume that
    \begin{itemize}
        \item $f: (V,C) \to T$ is a weak klt $(d,\Phi,{\leq}u, {\leq}v)$-SMM-type fibration such that $(X,C|_X)$ is $\epsilon$-lc for a general fiber $X$; 
        \item $f: (V,C) \to T$ satisfies condition $\mathcal{M}_1(t)$, so that we can take $0 \leq \Delta \sim_{\bQ} \delta(K_V + C)$ with $0 < \delta < t$, such that $X$ is a pure and exceptional  lc center of $(V,C + \Delta)$ by {\rm Lemma \ref{perturb}};
        \item by doing adjunction, we uniquely write 
        $$(K_V + C + \Delta)|_X = K_X + C|_X + B_X + {\mathbf{M}}_X,$$
        and get a generalized pair $(X,C|_X + B_X + \mathbf{M}_X)$ as in {\rm Lemma \ref{basic-for-adjunction-lc-centers}};
        \item $\Delta'_X \geq 0$ is an $\bR$-Cartier $\bR$-divisor on $X$ with the property that $t(K_X + C|_X) - \Delta'_X$ is pseudo-effective.
    \end{itemize}
    Then, the generalized pair $(X,C|_X + B_X +\Delta'_X + \mathbf{M}_X)$ is gklt.\par
    In particular, if $\Delta' \geq 0$ is an $\bR$-Cartier $\bR$-divisor on $V$ with the property that $t(K_V + C) - \Delta'$ is pseudo-effective and that the support of $\Delta'$ does not contain $X$, then $(V,C +\Delta + \Delta')$ is klt in a punctured neighborhood of $X$.
\end{thm}

\begin{proof}
    {\em Step 1.} In this step, we explore $(X,C|_X)$ and construct the number $t$. \par 
    By assumption, there is a birational model $(X,C|_X) \dashrightarrow (\overline{X}, C_{\overline{X}})$, such that $(X,C|_X)$ and $(\overline{X}, C_{\overline{X}})$ are crepant and that
    \begin{itemize}
        \item[(1)] $(\overline{X}, C_{\overline{X}})$ is semi-ample and gives a contraction $\overline{h}: \overline{X} \to Z$;
        \item[(2)] $A_{\overline{X}}$ is an integral divisor on $X$ that is big and semi-ample over $Z$ and $0 \leq \vol(A_{\overline{X}}|_F) \leq u$ holds for a general fiber $F$ of $\overline{h} : \overline{X} \to Z$;
        \item[(3)] $\ivol(K_{\overline{X}} + C_{\overline{X}}) \leq v$ holds. 
    \end{itemize}
    
    We do adjunction to write
    $$K_{\overline{X}} + C_{\overline{X}} \sim_\bQ h^*(K_{Z} + C_Z + \overline{\mathbf{M}}_Z),$$
    and get a generalized pair $(Z,C_Z + \overline{\mathbf{M}}_Z)$. Since $(X,C|_X)$ is crepant to $(\overline{X},C_{\overline{X}})$, the induced map $h: X \dashrightarrow Z$ is a morphism and is also a klt-trivial fibration. If we do adjuntion to $h : X \to Z$, we will get the same generalized pair. By Lemma \ref{bounded-base-for-smm}, there exist $r,\eta$ that depend on $d,\Phi,u,v,\epsilon$, and a very ample divisor $A_Z$ on $Z$ with ${A_Z}^{\dim Z} \leq r$ and $A_Z - K_Z$ $A_{Z} - C_Z$, $A_Z - \overline{\mathbf{M}}_Z$ are all pseudo-effective and $(Z,C_Z + \overline{\mathbf{M}}_Z)$ is generalized $\eta$-lc. By Lemma \ref{bounded-glct}, there is a number $t > 0$ such that for any $\bR$-Cartier $\bR$-divisor $D \geq 0$ with $A_Z - D$ being pseudo-effective, the generalized pair $(Z,C_Z + 12tD + \overline{\mathbf{M}}_Z)$ is gklt. \par
    \bigskip

    {\em Step 2.} In this step, we verify that $t$ satisfies all the required conditions. \par 
    By construction, we see that $B_X + \mathbf{M}_X \sim_{\bQ} \delta(K_X + C|_X)$ so that $h^*(3tA_Z) - (B_X + \mathbf{M}_X)$ is pseudo-effective. Also, we see $B_X \geq 0 $ by Lemma \ref{basic-for-adjunction-lc-centers}. Since $t(K_X + C|_X) - \Delta'_X$ is pseudo-effective, we get $h^*(3tA_Z) - \Delta '_X$ is pseudo-effective. Then, by \cite[Lemma 2.7]{CL24}, there is a $0 \leq B_{Z} \sim_\bR 6tA_Z$ such that $h^*(B_Z) - B_X \geq 0$ and there is a $ 0 \leq \Delta'_Z \sim_{\bR} 6tA_Z$ such that $h^*(\Delta'_Z) \geq \Delta'_X$. \par
    
    To prove that $(X,C|_X + B_X + \Delta'_X + \mathbf{M}_X)$ is gklt. By Lemma \ref{klt-property-for-adjunction}, it suffices to prove $(Z,C_Z + B_Z + \Delta'_Z  + \overline{\mathbf{M}}_Z)$ is gklt. This is directly from the construction of $t$. Finally, we know $(V,C + \Delta + \Delta')$ is klt in a punctured neighborhood of $X$ by inversion of adjunction \cite[Theorem 1.1]{FH23}.
\end{proof}

We next give a property of extremal contractions for fibrations of $\mathcal{M}_2(t)$-type.

\begin{thm}\label{Vertical-for-ext-locus}
    Let $d>0$ be an integer, $\Phi\subseteq [0,1]$ be a DCC set, $\epsilon,u,v > 0$ be real numbers. Then, there exists a real number $t > 0$ that satisfies the following. Assume that
    \begin{itemize}
        \item[(1)] $f : (V,C) \to T$ is a weak klt $(d, \Phi,{\leq}u,{\leq}v)$-SMM-type fibration such that $(X,C|_X)$ is $\epsilon$-lc for a general fiber $X$;
        \item[(2)] $f: (V,C) \to T$ satisfies condition $\mathcal{M}_2(t)$. 
    \end{itemize}
    Then, for any extremal contraction $\phi : V \to Y$ that contracts exactly an extremal ray $R \subseteq \overline{NE}(V)$ with $(K_V + C)\cdot R < 0$, the exceptional locus of $\phi$ must be vertical over $T$.
\end{thm}

\begin{proof}
    Assume conversly that $L$ is a horizontal$/T$ component of ${\rm Exc}(\phi)$. Define $Z := \phi(L)$. Let $S:=\phi^{-1}(q)$, where $q \in Z$ is a general point. Since $K_V + C$ is nef over the generic point of $T$ and $(K_V + C)\cdot R < 0$ and $L$ is $f$-horizontal, $S$ cannot be contracted to a point by $f$. Since $f:(V,C) \to T$ satisfies $\mathcal{M}_2(t)$, $L$ is $f$-horizontal and $q \in Z$ is general, we can choose two different general points $x_1,x_2 \in S$ with the properties that 
    \begin{itemize}
        \item there exist $ 0 \leq \Delta \sim_{\bR}\delta(K_V + C)$ with $\delta \in (0,t)\cap \bQ$, such that $X_2$ is a pure and exceptional lc center of $(V,C +\Delta)$ and $(V,C + \Delta)$ is not klt at $X_1$, where $X_1,X_2$ is the unique fiber of $f: V \to T$ passing through $x_1,x_2$ respectively.
    \end{itemize}
   
    By Theorem \ref{lc-property}, we can choose a small $t>0$ such that $(V,C +\Delta)$ is klt in a punctured neighborhood of $X_2$. On the other hand, since $-(K_V + C+\Delta)$ is $\phi$-ample, by the connectedness lemma \cite[Theorem 17.4]{Kol93}, the non-klt center of $(V,C + \Delta)$ should be connected along any fiber of $\phi$. However, along $S$, there is a non-klt center passing through $x_2$ and another non-klt center passing through $x_1$, which leads to a contradiction. 
\end{proof}

\begin{proof}[Proof of Theorem \ref{Res-for-Negative-parts}]
    Let $t>0$ be the number defined in Theorem \ref{Vertical-for-ext-locus} holds. We will prove $t$ satisfies the requirements. \par

    \bigskip
    
    {\em Step 1.} In this step, we run MMP$/T$ to construct a birational model $V'$.\par
    Since $f:(V,C) \to T$ admits a klt $(d,\Phi,{\leq}u, {\leq}v)$-SMM-type fibration, by Lemma \ref{MMP-inv-for-SMM-type}, we can run MMP over $T$ to get a weak klt $(d,\Phi,{\leq}u, {\leq}v)$-SMM-type fibration $f' : (V',C') \to T$. By Theorem \ref{MMP-inv-for-C(t)}, $f:(V',C') \to T$ also satisfies the condition $\mathcal{M}_2(t)$. Denote by $C_{X'}$ the restriction of $C'$ to a general fiber $X'$ of $V' \to T$. Take $p: V'' \to V$ and $q: V'' \to V'$ be a common log resolution. Since $K_{X'} + C_{X'}$ is semi-ample, we have $N_{\sigma}(K_{V'} + C';V'/T)$ is vertical over $T$ by Lemma \ref{Zariski-negative-in-family}.\par 
    $$\xymatrix{
        &  V''\ar[dl]_{p}\ar[dr]^{q} & \\
      V\ar@{-->}[rr]^{\varphi} \ar[dr] &  & V' \ar[dl] \\
        &  T  &
    }$$
    
    \bigskip
    
    {\em Step 2.} In this step, we prove that $N_{\sigma}(K_{V'} + C')$ is vertical over $T$. \par 
    Assume that $D$ is a component of $N_{\sigma}(K_{V'} + C')$.
    By the cone theorem, there is an extremal ray $R$ of $\overline{NE}{(V')}$ with $(K_{V'} + C') \cdot R < 0$ and curves in $R$ span $D$. By contraction theorem, there is a contraction morphism $\psi: V' \to Y'$ that contracts exactly the divisor $D$. By Theorem \ref{Vertical-for-ext-locus}, $D$ must be vertical over $T$. Thus, $N_{\sigma}(K_{V'} + C')$ is vertical over $T$.

    {\em Step 3.} In this Step, we finish the proof using the properties of MMP.\par
    By the negativity lemma \cite[Theorem 3.3]{Bir12}, we have 
    $$p^*(K_V +C) = q^*(K_{V'} + C') + E,$$
    where $E \geq 0$ is a $q$-exceptional divisor. Thus, we have
    $$N_\sigma(K_V + C) = p_*(q^*N_\sigma(K_{V'} + C') + E),$$
    and
    $$N_{\sigma}(K_V + C;V/T) =  p_*(q^*N_\sigma(K_{V'} + C';V'/T) + E).$$
    It follows that $N_{\sigma}(K_V + C) - N_\sigma(K_V + C;V/T)$ is vertical over $T$. After restricting to $X$, by Lemma \ref{Zariski-negative-in-family}, we get $N_\sigma(K_V + C)|_X = N_
    \sigma(K_X + C|_X)$. We finish the proof.
\end{proof}

\begin{proof}[Proof of Corollary \ref{Res-for Negative-parts-2}]
    This follows from Theorem \ref{Res-for-Negative-parts} and Lemma \ref{DCC-fixed-volume-imply-epsilon-lc}.
\end{proof}
Next, we give a corollary for pairs of log general type.
\begin{cor}
    Let $d >0$ be an integer, $\Phi \subseteq [0,1]$ be a DCC set, $\epsilon,v >0$ be real numbers. Then, there exists a real number $t > 0$ that satisfies the following. Assume that 
    \begin{itemize}
        \item $(V,C)$ is a projective klt $\bQ$-pair with $K_V + C$ big; 
        \item $f:V \to T$ is fibration. For a general fiber $X$ of $f$, $\dim X = d$, $\vol(K_X + C|_X) \leq v$ and $(X,C|_X)$ is $\epsilon$-lc;
        \item $f:(V,C) \to T$ satisfies condition $\mathcal{M}_2(t)$.
    \end{itemize}
    Then, we have 
    $$N_\sigma(K_V + C)|_X = N_
    \sigma(K_X + C|_X).$$
\end{cor}

\begin{proof}
    By \cite[Theorem 1.1]{BCHM10}, it is well-known that $(V,C)$ has a good minimal model. Thus, it follows from Theorem \ref{Res-for-Negative-parts}.
\end{proof}

\section{Extension Theorem for Special Fibrations}\label{Ext-thm-for-special-fibrations}

In this section, we will prove Theorem \ref{main-Extention-for-2-fibers-DCC}, an extension result for certain fibrations. We start with a simpler case where only one general fiber is considered. \par

\begin{thm}\label{extension-for-one-fiber-DCC}
    Let $d>0$ be integers, $\Phi \subseteq[0,1]$ be a DCC set, $\epsilon,u,v > 0$ be real numbers. Then there exists a real number $t>0$ that satisfies the following. Assume that
        \begin{itemize}
            \item[(1)] $f: (V,C) \to T$ admits a klt $(d,\Phi,{\leq}u, {\leq}v)$-SMM-type fibration and $(X,C|_X)$ is $\epsilon$-lc for a general fiber $X$; 
            \item[(2)] $f: (V,C) \to T$ satisfies contidion $\mathcal{M}_2(t)$;
            \item[(3)] $(V,C)$ has a good minimal model.  
        \end{itemize}
    Then, for any integer $m \geq 2$ with $mC$ being integral, there is a surjective map 
    $$H^0(V, m(K_V + C)) \to H^0(X, m(K_X  + C|_X)).$$
\end{thm}

\begin{proof}
 
    {\em Step 1.} In this step, we make some reduction for $f : (V,C) \to T$. \par 
    Since $f: V \to T$ admits a klt $(d,\Phi,{\leq}u, {\leq}v)$-SMM-type fibration, by Lemma \ref{MMP-inv-for-SMM-type}, we can run MMP$/T$ to get a weak klt $(d,\Phi,{\leq}u, {\leq}v)$-SMM-type fibration $f' :(V',C') \to T$. Take $p:V'' \to V$ and $q:V'' \to V'$ be a common log resolution. Denote by $X'$ the birational transform of the general fiber $X$. By Theorem \ref{MMP-inv-for-C(t)}, $f': (V',C') \to T$ also satisfies $\mathcal{M}_2(t)$. We can write 
    $$p^*(K_V + C) = q^*(K_{V'} + C') + G$$
    for some $q$-exceptional divisor $G \geq 0$. Thus, for any integer $m \geq 2$ with $mC$ being 
    integral, we see that 
    \begin{align*}
        H^0(V, m(K_V + C)) &\cong H^0(V', m(K_{V'} + C')), \\
        H^0(X,m(K_X + C|_X)) &\cong H^0(X', m(K_{X'} + C|_{X'}))
    \end{align*}
    and the restriction map $$H^0(V, m(K_V + C)) \to H^0(X, m(K_X  + C|_X))$$
    is just the same as the restriction map 
    $$H^0(V', m (K_{V'} + C')) \to H^0(X',m(K_{X'} + C'|_{X'})),$$
    where $X'$ is the birational transform of $X$. Therefore, We can replace $f:(V,C) \to T$ with $f':(V',C') \to T$ and assume that $f:(V,C) \to T$ is a weak klt $(d,\Phi,{\leq}u, {\leq}v)$-SMM-type fibration at the beginning.
    
    \bigskip

    {\em Step 2.} In this step, we construct the number $t$.\par
    By Theorem \ref{Res-for-Negative-parts} and Theorem \ref{lc-property}, there exists a number $t > 0$ such that there exists $0 \leq \Delta \sim_{\bQ}\delta(K_V + C)$ for some $\delta \in (0,t) \cap \bQ$ such that $X$ is a pure and exceptional lc center of  $(V, C + \Delta)$, and $(V, C + \Delta)$ is klt in a punctured neighborhood of $X$ and
    $$N_{\sigma}(K_V + C)|_X = N_{\sigma}(K_X + C|_X).$$
    It remains to check that $t$ satisfies all the required conditions.

    \bigskip
    
    {\em Step 3.} In this step, we take a log resolution of $(V, C + \Delta)$ and give further notations. \par
    Let $\nu: W \to V$ be a log resolution such that the unique lc place $E$ lying over $X$ is a divisor on $W$. By the connectedness lemma \cite[Theorem 17.4]{Kol93}, $g:= \nu|_{E} : E \to X$ is a contraction. We can write uniquely
    $$\nu^*(K_V + C + \Delta) = K_W + E + \Gamma + \Theta - F,$$
    with the property that 
    \begin{itemize}
        \item $\Gamma,\Theta,F \geq 0$ have no common components;
        \item any component of $\Gamma|_E$ is $g$-horizontal and any component of $\Theta|_E$ is $g$-vertical;
        \item all components of $E$, $\Gamma$, $\Theta$, $F$, $N_\sigma(\nu^*(K_V + C))$ have simple normal crossing support;
        \item $P_\sigma(\nu^*(K_V + C))$ is semi-ample since $(V,C)$ has a good minimal model.
    \end{itemize}
    For the convenience, we denote $\Gamma|_E, \Theta|_E,F|_E$ by $\Gamma_E, \Theta_E, F_E$ respectively. By Theorem \ref{lc-property} and the construction of $t$, we see that the sub-pair
    $$(E,\Gamma_E + \Theta_E - F_E)$$
    is sub-klt.
    \bigskip

    {\em Step 4.} In this step, we will get a surjective map of global sections using Theorem \ref{kollars-injectivity}. \par 
    Let $P :=  \fraction{\nu^*(m(K_V + C))}$ be the fractional part. Then, $P$ is $\nu$-exceptional since $mC$ is an integral divisor on $V$. We define 
     $$ L := \rounddown{\Gamma + \Theta -F -P+ (m-1-\delta )N_{\sigma}(\nu^*(K_V + C))}.$$
    Note that the support of $L$ does not contain $E$.\par 
    Since
    \begin{align*}
        &\qquad \rounddown{m(\nu^*(K_V + C))} - E - K_W - L \\
        &\sim_{\bQ} (m - 1 - \delta)\nu^*(K_V + C) + \Gamma  + \Theta - F - P - L\\
        &\sim_{\bQ} (m - 1 - \delta)P_\sigma(\nu^*(K_V + C)) \\
        &\hspace{10mm} +\fraction{\Gamma + \Theta - F - P + (m - 1 -\delta)N_{\sigma}(\nu^*(K_V + C)},
    \end{align*}
    it is a sum of a semi-ample $\bQ$-divisor and a fraction term. We see that our situation satisfies all the condition of Theorem \ref{kollars-injectivity}. Then by Theorem \ref{kollars-injectivity}, we get a surjective map 
    \begin{equation*}
        H^0(W, \rounddown{m(\nu^*(K_V + C))}-L) \to H^0(E, \rounddown{m( g^*(K_X + C|_X) )}- L|_E). 
    \end{equation*}
    \bigskip
    
    {\em Step 5.} In this step, we investigate the integral divisor $L|_E$. \par 
    Since $K_X + C|_X$ is nef, by Theorem \ref{Res-for-Negative-parts}, we have 
    $$N_\sigma(K_V + C)|_X = N_
    \sigma(K_X + C|_X)= 0.$$
    Pulling back to $W$, since $K_X+ C|_X$ is semi-ample, by Lemma \ref{property-for-NZ-decompose}, we see that 
    $$N_\sigma(\nu^*(K_V + C))|_E = N_\sigma(g^*(K_X + C|_X)) = 0.$$
    Thus, 
    \begin{align*}
        L|_E &= \rounddown{ \Gamma_E + \Theta_E - F_E -P|_E + (m - 1 - \delta)N_\sigma(\nu^*(K_V + C))|_E}  \\
        &= \rounddown{\Gamma_E + \Theta_E - F_E - P|_E}.
    \end{align*}
    Since $(E,\Gamma_E + \Theta_E - F_E)$ is log smooth and sub-klt, we see $L|_E \leq 0$. 

    \bigskip
    
    {\em Step 6.} 
    In this step, we finish the proof by chasing the diagram below. \par   
    $$\xymatrix{
        H^0(\rounddown{m \nu^*(K_V + C)} - L ) \ar@{^(->}[d] \ar@{->>}[r]
				& H^0(\rounddown{m g^*(K_X + C|_X)} - L|_E) \ar@{^(->}[d]\\
        H^0(\rounddown{m \nu^*(K_V + C)} + \roundup{F + P}) \ar@{=}[d] \ar[r]
				& H^0(\rounddown{m g^*(K_X + C|_X)} + \roundup{F_E + P|_E}) \\		
        H^0(\rounddown{m \nu^*(K_V+ C)}) \ar@{=}[d] \ar[r] ^r 
				& H^0(\rounddown{m g^*(K_X + C|_X)}) \ar@{=}[d] \ar@{^(->}[u] \ar@/_9.5pc/[uu]^{j}\\
        H^0(m(K_V + C)) \ar[r] 
                & H^0(m(K_X + C|_X )		
    }$$ \par
    In the above diagram, we need to mention that:
    \begin{itemize}
        \item the first row is surjective by Step 4;
        \item the injectivity of the vertical map between first and second line is due to fact that $-L \leq \roundup{F + P}$ and $-L|_E \leq \roundup{F_E + P|_E}$;
        \item the equality in the left sidd between second line and third line is because $F$ and $P$ are $\nu$-exceptional;
        \item the map $j$ is a well-defined injective map induces by the effective divisor $-L|_E$.
    \end{itemize}
    Combining all the results above, a diagram chase yields the surjectivity of the map in the bottom row of the diagram.
\end{proof}

Next, we prove the extention theorem for two general fibers.
\begin{thm}\label{Extention-for-2-fibers-DCC}{\rm (=Theorem \ref{main-Extention-for-2-fibers-DCC})}
    Let $d>0$ be an integer, $\Phi\subseteq[0,1]$ be a DCC set, $u,v,\epsilon>0$ be real numbers. Then, there exists a number $t>0$ that satisfies the following. Assume that
        \begin{itemize}
            \item[(1)] $f: (V,C) \to T$ admits a klt $(d,\Phi,{\leq}u, {\leq}v)$-SMM-type fibration such that $(X,C|_X)$ is $\epsilon$-lc for a general fiber $X$; 
            \item[(2)] $f: (V,C) \to T$ satisfies contidion $\mathcal{M}_2(t)$;
            \item[(3)] $(V,C)$ has a good minimal model.
        \end{itemize}
    Then, for any integer $m \geq 2$ with $mC$ being integral, there is a surjective map 
    $$H^0(V, m(K_V +C)) \to H^0(X_1, m(K_{X_1} + C|_{X_1})) \oplus H^0(X_2, m(K_{X_2} + C|_{X_2})),$$
    where $X_1,X_2$ are two different general fibers.
\end{thm}

\begin{proof}  

    {\em Step 1.}  In this step, we rewrite the condition $\mathcal{M}_2(t)$ and make some reduction. \par 

    By the definition of condition $\mathcal{M}_2(t)$, after possibly switching $X_1$ and $X_2$, we can assume that 
    \begin{itemize}
        \item[(1)] there exists a $\bQ$-divisor $0 \leq \Delta_1 \sim_{\bQ} \delta_1(K_V + C)$ for some $\delta_1 \in (0,t)\cap \bQ$, such that $X_1$ is a pure and exceptional lc center of $(V, C +\Delta_1)$ with unique lc place $E_1$ over $X_1$;
        \item[(2)] there exists a $\bQ$-divisor $0 \leq \Delta \sim_{\bQ} \delta(K_V + C)$ for some $\delta \in (0,t) \cap \bQ$, such that $X_2$ is a pure and exceptional lc center of $(V,C + \Delta)$ and $a(E_1,V,C + \Delta) \leq 0$ holds.
    \end{itemize}

    Take a log resolution $\nu: W \to V$ as in Step 4 in proof of Theorem \ref{extension-for-one-fiber-DCC} such that $E_1$ is a divisor on $W$ and denote by $g_1 : E_1 \to X_1$ the restriction of $\nu$ to $E_1$. We see that $g_1$ is a contraction. \par
By Theorem \ref{extension-for-one-fiber-DCC}, we get a surjective map 
\begin{equation}
    H^0(W, \nu^*(m(K_V + C))) \to H^0(E_1,g^*(m(K_{X_1} + C|_{X
    _1}))). \label{q1}
\end{equation}

Denote by $E_2$ the unique lc center of $(V, C + \Delta)$ lying over $X_2$. After replacing $\nu : W \to V$ with a further resolution, we assume $E_2$ is also a divisor on $W$ and denote by $g_2 : E_2 \to X_2$ the restriction of $\nu$ to $E_2$. we can assume that $a(E_1,V,C+\Delta)\leq 0$.\par
Note that the surjectivity of the restriction map
    $$H^0(m(K_V +C)) \to H^0(m(K_{X_1} + C|_{X_1})) \oplus H^0(m (K_{X_2} + C|_{X_2}))$$
    is equivalent to the surjectivity of 
    \begin{equation}
        H^0(\rounddown{\nu^*(m(K_V + C))})\to \bigoplus_{i=1,2} H^0(E_i, \rounddown{g_i^*(m (K_{X_i} + C|_{X_i}))}). \label{q2}
    \end{equation}

    Combining with \eqref{q1}, we only need to prove that for any section $u$ in $$H^0(E_2, \rounddown{g_2^*(m (K_{X_2} + C|_{X_2}))}),$$ the element $(0, u)$ in the right-hand-side of \eqref{q2} has a pre-image. In other words, we only need to prove 
    $$H^0(W, \rounddown{\nu^*(m(K_V+C))} - E_1) \to H^0(E_2, \rounddown{g_2^*(m (K_{X_2} + C|_{X2}))})$$
    is surjective. \par
    \bigskip

    {\em Step 2.} In this step, we prove the above surjective map. \par
    We write uniquely
    $$\nu^*(K_V + C + \Delta) = K_W + E_2 + sE_1 + \Gamma  + \Theta - F,$$
    where the following properties holds :
    \begin{itemize}
        \item $\Gamma,\Theta,F \geq 0$ have no common components and do not contain $E_1$, $E_2$ as components;
        \item any component of $\Gamma|_{E_2}$ is $g$-horizontal and any component of $\Theta|_{E_2}$ is $g_2$-vertical;
        \item all components of $E_1$, $E_2$, $\Gamma$, $\Theta$, $F$, $N_\sigma(\nu^*(K_V + C))$ have simple normal crossing supports. 
    \end{itemize}
    We denote $\Gamma|_{E_2}$, $\Theta|_{E_2}$, $F|_{E_2}$ by $\Gamma_{E_2}$, $\Theta_{E_2}$, $F_{E_2}$ respectively. 
    Let $P :=  \fraction{\nu^*(m(K_V + C))}$, the fractional part, which is $\nu$-exceptional and define
    $$L := \rounddown{\Gamma + \Theta + sE_1-F - P + (m -1-\delta )N_{\sigma}(\nu^*(K_V + C))}.$$
    As in Step 4 of proof of Theorem \ref{extension-for-one-fiber-DCC}, using Theorem \ref{kollars-injectivity}, we also get the surjective map
    $$H^0(W, \rounddown{m(\nu^*(K_V + C))}-L) \to H^0(E_2, \rounddown{m( g_2^*(K_{X_2} + C|_{X_2}))} - L|_{E_2}).$$

    Since $s \geq 1$ by $a(E_1,V,C + \Delta) \leq 0$,  we have $- L \leq -E_1 + \roundup{F + P} $. It follows that
    $$H^0( \rounddown{m \nu^*(K_V + C) }- L) \cong H^0(\rounddown{m \nu^*(K_V + C)} - E_1 + \roundup{F + P}).$$ 
    
    Finally, we can finish the proof by chasing the following diagram: 
        $$\xymatrix{
        H^0(\rounddown{m \nu^*(K_V + C)} - L ) \ar@{^(->}[d] \ar@{->>}[r]
				& H^0(\rounddown{m g_2^*(K_{X_2} + C|_{X_2})} - L|_{E_2}) \ar@{^(->}[d]\\
        H^0(\rounddown{m \nu^*(K_V + C)}  - E_1 + \roundup{F + P}) \ar@{=}[d] \ar[r]
				& H^0(\rounddown{m  g_2^*(K_{X_2} + C|_{X_2})} + \roundup{F_{E_2} + P|_{E_2}}) \\		
        H^0(\rounddown{m  \nu^*(K_V+ C)}- E_1) \ar[r] ^r 
				& H^0(\rounddown{m  g_2^*(K_{X_2} + C|_{X_2})})  \ar@{^(->}[u]
    }$$ \par
    Here we need to mention that the equality in the bottom of left-hand-side is because both $E_1$ and $F$ are $\nu$-exceptional and $E_1$ and $F + P$ have no common components.
\end{proof}

We end this section with an application to varieties of general type. \par

\begin{cor}
    Let $d>0$ be an integer, $\Phi \subseteq[0,1]$ be a DCC set, $\epsilon,v>0$ be real numbers. Then, there exists a positive number $t$ that satisfies the following. Assume that
    \begin{itemize}
        \item[(1)] $(V,C)$ is a projective klt $\bQ$-pair with $K_V + C$ big; 
        \item[(2)] $f:V \to T$ is a fibration. For a general fiber $X$, $\dim X = d$, $\vol(K_X + C|_X) \leq v$ and $(X,C|_X)$ is $\epsilon$-lc;
        \item[(3)] $f:(V,C) \to T$ satisfies condition $\mathcal{M}_2(t)$;
    \end{itemize}
    Then for any integer $m \geq 2$ with $mC$ being integral, there is a surjective map 
    $$H^0(m(K_V +C)) \to H^0(m(K_{X_1} + C|_{X_1})) \oplus H^0(m (K_{X_2} + C|_{X_2})),$$
    where $X_1,X_2$ are two different general fibers.
\end{cor}

\begin{proof}
    This is immediate from Theorem~\ref{Extention-for-2-fibers-DCC} and \cite[Theorem 1.1]{BCHM10}.
\end{proof}

\section{Proofs and Applications.}\label{Appications}
    This section is devoted to proving main results. As preparation, we first establish Theorem~\ref{main-Effective-Iitaka}.

    \begin{proof}[Proof of Theorem \ref{main-Effective-Iitaka}.]
        Let $h : V \to Y$ be the contraction defined by the semi-ample divisor $K_V + C$ and write 
        $$K_V + C \sim_{\bR}h^*(K_Y + B_Y + M_Y)$$
        as in Section \ref{adjunction} to get a generalized pair $(Y,B_Y + \mathbf{M}_Y)$.\par
        By the proof of \cite[Theorem 7.4]{Bir21b} or \cite[Theorem 3.2]{Zhu25}, there exists an integer $q>0$, a DCC set $\Psi$, and a finite set $I$, such that $$K_V + C \sim_{q}h^*(K_Y + B_Y + M_Y),$$
        the coefficients of $B_Y$ belong to $\Psi$ and $M_Y = \sum_{i}{p_i M_{Y,i}}$, where all $M_{Y,i}$ are b-Cartier-nef b-divisors and $p_i \in I$. Thus, for any integer $r > 0$, we have 
        $$H^0(\rounddown{rq(K_V + C)}) \cong H^0(\rounddown{rq(K_Y + B_Y + M_Y)}).$$
        By \cite[Theorem 8.2]{BZ16}, there exists a number $m_0$ that depends only on $n,\Psi,I$ (hence only depends on $n,\Phi ,u$) such that $|\rounddown{m_0(K_Y + B_Y + M_Y)}|$ gives a birational map. Replacing $m_0$ with $qm_0$, we see $m_0$ satisfies the requirement.
    \end{proof}
        
    Here we recall that for any integer $n>0$, DCC set $\Phi \subseteq[0,1]$ and $u >0$, we denote by $r(n,\Phi,u)$ the minimal integer $m_0>0$ such that Theorem \ref{main-Effective-Iitaka} holds. \par

    \begin{proof}[Proof of Theorem \ref{main-for-Iitaka-large-volume}]
        By Theorem \ref{Extention-for-2-fibers-DCC}, for any integer $d>0$ and real number $v>0$, there is a uniform number $t(d,\Phi,u,v,\epsilon)$ that satisfies all the required conditions in Theorem \ref{Extention-for-2-fibers-DCC}, where $\Phi,u,\epsilon$ are given data. We construct a function 
        \begin{align*}
            \lambda: \bZ_{>0} \times \bR_{>0} &\to \bR_{>0} \\
                                (d, v)    &\mapsto t(d,\Phi,u,v,\epsilon).
        \end{align*}
        For the function $\lambda$, Theorem \ref{construct-fibration} now gives constants $v_{n-1}>\dots>v_1>0$ and $v_0>0$ such that for any $(V,C)$ that admits a klt $(n,\Phi,{\leq}u,{\geq}v_0)$-stable minimal model, after replacing $(V,C)$ with a crepant pair,  we are in one of the following two situations:

        \begin{itemize}
            \item[(1)] $|m(K_V + C)|$ induces an Iitaka fibration since $m \geq 2$ and $mC$ is integral;
            \item[(2)] We get a fibration $f:(V,C) \to T$ that satisfies $\mathcal{M}_2(\lambda(k,v_k))$ and $\ivol(K_X + C|_X) \leq v_k$ for a general fiber $X$.
        \end{itemize}
        In Case (1), there is nothing more to prove. In case (2), $f:(V,C) \to T$ admits a klt $(k,\Phi,{\leq}u, {\leq}v_k)$-SMM-type fibration. Since $m$ is divisible by $r(k,\Phi,u)$, the linear system $|m(K_X + C|_X)|$ induces an Iitaka fibration on a general fiber $X$. Hence, Theorem \ref{Extention-for-2-fibers-DCC} implies that $|m(K_V + C)|$ induces an Iitaka fibration on $V$.  
    \end{proof}

    We next apply the result to pairs of log general type. In the general-type case, the parameter $u$ is not involved, so we abbreviate $r(n,\Phi,u)$ to $r(n,\Phi)$.

    \begin{cor}\label{general-type}
        Let $n>0$ be an integer and $\Phi \subseteq [0,1]$ be a DCC set, $\epsilon>0$ be a real number. Then, there exists a real number $v>0$ that satisfies the following. Assume that
        \begin{itemize}
            \item $(V,C)$ is an $\epsilon$-lc projective pair with the coefficients of $C$ belong to $\Phi$;
            \item $\vol(K_V +C) \geq v$.
        \end{itemize}
        Then, for any $m > 0$ such that $mC$ is integral and $m$ is divisible by $r(k,\Phi)$ for any $1 \leq k \leq n-1$, $|m(K_V + C)|$ induces a birational map.        
    \end{cor}

    \begin{proof}
        It follows immediately from Theorem~\ref{main-for-Iitaka-large-volume}.
    \end{proof}

    \begin{proof}[Proof of Theorem \ref{main-Effective-Iitaka-no-boundary}]
        After replacing $V$ with its good minimal model, we can assume that $(V,0)$ is a klt $(n,\{0\},u)$-stable minimal model at the beginning. \par
        Let $h: V \to Y$ be the contraction defined by the semi-ample divisor $K_V$ and write 
        $$K_V \sim_{\bQ} h^*(K_Y + B_Y + M_Y)$$
        as in Section \ref{adjunction} to get a generalized pair $(Y,B_Y + \mathbf{M}_Y)$. We define $D:= K_Y + B_Y + M_Y$. Note that $D$ is an ample $\bQ$-divisor. \par
        As in the proof of Theorem \ref{main-Effective-Iitaka}, we know that there exists a number $m_0$ that depends only on $n,u$, such that $|\rounddown{m_0 D}|$ gives a birational map. By \cite[Lemma 2.3.4 (2)]{HMX13}, we see $m'D$ is potentially birational for any integer $m' \geq (2n+1)m_0$. Thus, by Theorem \ref{app-for-kol-inj}, we get $|mK_V|$ induces an Iitaka fibration for any integer $m \geq (2n+1)m_0 + 1$. We finish the proof.
    \end{proof}

    We recall that for any integer $n>0$ and real number $u > 0$, $s(n,u)$ denotes the minimal integer $s_0>0$ for which Theorem \ref{main-Effective-Iitaka-no-boundary} holds. \par

    \begin{proof}[Proof of Theorem \ref{main-no-boundary}]
        When $m$ is divisible by $r(k,\{0\},u)$ for any $1 \leq k \leq n-1$, it follows by Theorem \ref{main-for-Iitaka-large-volume}. When $m \geq s(n-1,u)$, it follows by Theorem \ref{construct-fibration}, Theorem \ref{Extention-for-2-fibers-DCC} as in the proof of Theorem \ref{main-for-Iitaka-large-volume}.
    \end{proof}
    
\section{Examples}\label{examples}
This section presents several illustrative examples. The first example shows that in Theorem \ref{Res-for-Negative-parts}, \ref{extension-for-one-fiber-DCC} and \ref{Extention-for-2-fibers-DCC}, the number $t$ must depend on  $v$, the upper bound for the Iitaka volume of a general fiber. In other words, one cannot expect to get a number $t$ that works uniformly for all $v$. 
\begin{eg}\em 
    Let $f:V \to T$ be the Hirzebruch surface $\mathbb{F}_1$ over $\bP^1$. Let $X$ be a general fiber, then $X \cong \bP^1$. Let $E \subseteq V$ be the unique section with negative self-intersection. Let $\sigma : V \to \bP^2 $ be the blow-down morphism. Write $H$ for the pullback of a hyperplane section from $\bP^2$. \par 
    Take $l\geq 6$ to be an integer. Let $H_1,\cdots,H_{l}$ be general elements in $|H|$. Let $C:= \frac{1}{2}\sum_{i=1}^l H_i + \frac{1}{2} E$. We claim that $$H^0(2(K_V + C)) \to H^0(2(K_X + C|_X))$$
    is never surjective.\par
    Indeed, $K_V \sim -3H + E$ gives $2(K_V+C)\sim (l-6)H+3E$, so the Nakayama-Zariski negative part is $N_\sigma(2(K_V+C))=3E$. Hence, every effective divisor linearly equivalent to $2(K_V+C)$ contains $E$ as a component. On the other hand, $\deg(2(K_X+C|_X))=l-3>0$, so the linear system $|2(K_X+C|_X)|$ is base-point-free. Consequently, 
    $$N_\sigma(K_V +C) |_X \ne N_\sigma(K_X + C|_X)$$ and the restriction map cannot be surjective.\par
    
    In fact, when $l$ tends to infinity, any number $t>0$ for which Theorems \ref{Res-for-Negative-parts}, \ref{extension-for-one-fiber-DCC} and \ref{Extention-for-2-fibers-DCC} hold must depend on $l$ here and tend to zero.
\end{eg}

Next example shows that in Theroem \ref{extension-for-one-fiber-DCC} and \ref{Extention-for-2-fibers-DCC}, we cannot weaken the condition ``$f :(V,C) \to T$ satisfies $\mathcal{M}_{2}(t)$" to ``there is a $ 0 \leq \Delta \sim_{\bQ}t(K_V + C)$ with $\mult_{\eta_{X_i}}(\Delta) \geq n$ for $i=1,2$".

\begin{eg}\em
    Let $X$ be the Hirzebruch surface $\mathbb{F}_1$ with a section curve $C$ and a fiber $F$. We know that $(C^2) = -1$, $(F^2) = 0$ and $(C\cdot F) =1$. We see $K_X \equiv -2C - 3F$. Since $6C + 7F$ is ample, we can take a general $H \in |6C + 7F|$, and let $L_X:= \frac{1}{2}H$. Then $(X,L_X)$ is klt. By calculation, we see $2(K_X + L_X) \sim 2C + F$. Hence, $\vol(K_X + L_X)$ is bounded and
    $$H^0(X, 2(K_X + L_X)) \to H^0(F,2(K_F + L_X|_F))$$
    cannot be surjective. \par
    Let $T = \bP^1$ and $V:= X \times T$. Denote by $p_1 : V \to X$ and $p_2: V \to T$ the projection maps. Let $L':= p_1^*(L_X)$. For any positive integer $l$, take $2l + 4$ different general points $\{q_i\}_{i=1}^{2l+4}$ on $T$ and denote divisors $p_2^*(\{q_i\})$ by $X_i$. Let $L:= L' + \frac{1}{2}\sum_{i = 1}^{2l+4} X_i$. Then we see that there is a $ 0 \leq \Delta \sim_{\bQ} \frac{n}{l}(K_V + L)$ with $\mult_{\eta_X}(\Delta) \geq n$. However, the restriction map 
    $$H^0(V, 2(K_V +L) ) \to H^0(F, 2(K_F + L|_F))$$
    is never surjective for any integer $l >0$.  
\end{eg}


\end{document}